\documentclass[a4paper, 11pt]{amsart}

\usepackage[french, english]{babel} 
\usepackage[T1]{fontenc}
\usepackage[ansinew]{inputenc}

\title{Hypoelliptic estimates for linear transport operators}

\author{Paul Alphonse}
\address{Paul Alphonse, Univ Rennes, CNRS, IRMAR - UMR 6625, F-35000 Rennes}
\email{paul.alphonse@ens-rennes.fr}

\date{}



\usepackage[top=3cm, bottom=2cm, left=3cm, right=3cm]{geometry}		

\usepackage{enumitem}

\usepackage{bbm}

\usepackage{array}								

\usepackage{amsfonts}
\usepackage{amsmath}
\usepackage{amsthm}
\usepackage{amssymb}

\usepackage{stmaryrd}

\usepackage{comment}

\usepackage[dvipsnames]{xcolor}

\usepackage{hyperref}

\hypersetup{	
colorlinks=true,
breaklinks=true,
urlcolor= red,
linkcolor= red,
citecolor=ForestGreen
}

\usepackage[scr]{rsfso}

\usepackage{comment}

\numberwithin{equation}{section}

\newtheorem{thm}{Theorem}[section]
\newtheorem{prop}[thm]{Proposition}
\newtheorem{lem}[thm]{Lemma}

\theoremstyle{definition}

\newtheorem{ex}[thm]{Example}

\DeclareMathOperator{\Supp}{Supp}
\DeclareMathOperator{\Reelle}{Re}
\DeclareMathOperator{\Op}{Op}
\DeclareMathOperator{\Tr}{Tr}
\DeclareMathOperator{\Ker}{Ker}
\DeclareMathOperator{\Ran}{Ran}

\DeclareMathOperator{\Rank}{Rank}

\begin{document}

\sloppy

\maketitle

\selectlanguage{english}

\begin{abstract} We aim at understanding how the non-commutation phenomena between a linear transport operator and a fractional diffusion allow the transport operator to satisfy hypoelliptic estimates on the whole space. Such hypoelliptic estimates are obtained for a large class of linear transport operators by using the classical multiplier method. The main motivation of this work arises from the study of the hypoelliptic regularity of the solutions of kinetic equations associated with a free transport operator.
\end{abstract}

\section{Introduction}

This paper is devoted to study the hypoelliptic properties enjoyed by a large class of linear transport operators, time-dependent or autonomous, given by
\begin{equation}\label{05042019E6}
	L = \partial_t + Bx\cdot\nabla_x\quad\text{or}\quad L = Bx\cdot\nabla_x,\quad (t,x)\in\mathbb R\times\mathbb R^n,
\end{equation}
where $B$ is a real $n\times n$ matrix and $Bx\cdot\nabla_x$ stands for the following differential operator
$$Bx\cdot\nabla_x = \sum_{i=1}^n\sum_{j=1}^nB_{i,j}x_j\partial_{x_i},\quad B = (B_{i,j})_{1\le i,j\le n}.$$
Considering $\lambda_0>0$ a positive real number and $Q$ a real symmetric positive semidefinite $n\times n$ matrix, satisfying an algebraic condition with the matrix $B$, we aim at studying how the non-commutation phenomena between the transport operator $L$ and the fractional diffusion $\langle QD_x\rangle^{\lambda_0}$ allow the operator $L$ to satisfy global hypoelliptic estimates on the whole space. As justified just after, the estimates we investigate take the following form
\begin{equation}\label{06052020E1}
	\sum_{j=1}^{r-1}\big\Vert\langle Q(B^T)^jD_x\rangle^{\lambda_j}u\big\Vert_{L^2} + \Vert u\Vert_{H^{\lambda_r}}\lesssim\big\Vert\langle QD_x\rangle^{\lambda_0}u\big\Vert_{L^2} + \sum_{j=0}^{r-1}\big\Vert\langle Q(B^T)^jD_x\rangle^{q_j}Lu\big\Vert_{H^{s_j}},
\end{equation}
with $D_x = i^{-1}\nabla_x$ and where $0\le r\le n-1$ is an integer intrinsically linked to the matrices $B$ and $Q$ and $\lambda_1,\ldots,\lambda_r>0$ are regularity exponents we aim at sharply explicit in terms of $\lambda_0>0$ and the non-negative real numbers $0\le q_0,\ldots, q_{r-1}\le 1$ and $s_0,\ldots,s_{r-1}\geq0$. We refer to the notation paragraph just after where the norms appearing in \eqref{06052020E1} are defined.

The algebraic condition mentioned above is the so-called Kalman rank condition
\begin{equation}\label{10052018E4}
	\Rank\big[B\ \vert\ Q\big] = n,
\end{equation}
where 
$$\big[B\ \vert\ Q\big] = \big[Q,BQ,\ldots,B^{n-1}Q\big],$$
is the $n\times n^2$ matrix obtained by writing consecutively the columns of the matrices $Q, BQ,\ldots,B^{n-1}Q$. This condition appears in many branches of analysis, for example in controllability theory or in the theory of hypoelliptic operators. In this last domain, the Kalman rank condition is known to be one of the equivalent conditions that characterizes the hypoellipticity of the Ornstein-Uhlenbeck operators, which are operators given by the sum of a diffusion and a linear transport operator, see e.g. the introduction of \cite{AB}. This condition also appears in the study of the subelliptic properties enjoyed by a more general class of operators known as the fractional Ornstein-Uhlenbeck operators and defined by
$$P = \vert QD_x\vert^{\lambda_0} + Bx\cdot\nabla_x,\quad x\in\mathbb R^n.$$
The Ornstein-Uhlenbeck operators correspond to the case where $\lambda_0 = 2$. In the work \cite{AB} (Corollary 1.17), by studying the smoothing properties of the semigroup generated by the above operator $P$ and using results from the theory of real interpolation, the author and J. Bernier obtained subelliptic estimates enjoyed by the operator $P$ under the Kalman rank condition \eqref{10052018E4} that directly provide the following hypoelliptic estimates for the autonomous transport operator $L$ under the same setting
\begin{equation}\label{06052020E2}
	\sum_{j=1}^{r-1}\big\Vert\langle Q(B^T)^jD_x\rangle^{\lambda_j}u\big\Vert_{L^2} + \Vert u\Vert_{H^{\lambda_r}}\lesssim\big\Vert\langle QD_x\rangle^{\lambda_0}u\big\Vert_{L^2} + \Vert Lu\Vert_{L^2},
\end{equation}
where $0\le r\le n-1$ is again an integer intrinsically linked to the matrices $B$ and $Q$ (the same as above) and the regularity exponents $\lambda_j>0$ are given by
\begin{equation}\label{06052020E5}
	\lambda_j = \frac{\lambda_0}{1+j\lambda_0}.
\end{equation}
Therefore, when the Kalman rank condition holds, the non-trivial interaction between the operator $L$ and the fractional diffusion $\langle QD_x\rangle^{\lambda_0}$ provides global Sobolev regularity in the all space and also in the directions given by the matrices $Q,QB^T,\ldots,Q(B^T)^{n-1}$. In this paper, we aim at investigating the influence of the possible regularity of the source term $Lu$ in the same specific directions on the regularity of the function $u$. The purpose is to track the dependence of the regularity exponents $\lambda_j>0$ involved in the hypoelliptic estimate \eqref{06052020E1} with respect to $\lambda_0>0$, $0\le q_0,\ldots, q_{r-1}\le 1$ and $s_0,\ldots,s_{r-1}\geq0$.

The main motivation of this work arises from the study of the hypoelliptic regularity of the $L^2_{t,x,v}$ solutions of the following transport equation with a source term $f\in L^2_{t,x,v}$, with $L^2_{t,x,v} = L^2(\mathbb R^{2n+1}_{t,x,v})$, and posed on the whole Euclidean space,
\begin{equation}\label{11042019E1}
	\partial_tu + v\cdot\nabla_xu = f,\quad (t,x,v)\in\mathbb R\times\mathbb R^n\times\mathbb R^n.
\end{equation}
Notice that this is the equation $Lu = f$, where $L$ is the time-dependent transport operator defined in \eqref{05042019E6} when the matrix $B$ is given by 
$$B = \begin{pmatrix}
	0_n & I_n \\
	0_n & 0_n
\end{pmatrix}.$$
Many results on this topic concerning this particular equation were obtained in a series of works by F. Bouchut \cite{MR1949176}, R. Alexandre, Y. Morimoto, S. Ukai, J.-C. Xu and T. Yang \cite{MR2462585}, R. Alexandre \cite{MR2948893} and W.-X. Li, P. Luo, S. Tang in \cite{MR3456819}. One of these results, namely \cite{MR3456819} (Theorem 1.1), provide the following hypoelliptic estimate for the solutions $u\in L^2_{t,x,v}$ of the transport equation \eqref{11042019E1}
\begin{align}
	& \big\Vert\langle D_x\rangle^{\frac p{1-q+p}}u\big\Vert_{L^2_{t,x,v}}\lesssim\big\Vert\langle D_v\rangle^pu\big\Vert_{L^2_{t,x,v}} + \big\Vert\mathcal\langle D_v\rangle^qf\big\Vert_{L^2_{t,x,v}}, \label{15072019E2} 
\end{align}
where $p\geq0$ and $0\le q\le 1$ are non-negative real numbers. Let us also mention that many results in the litterature known as {\og}Avering Lemmas{\fg} deal with the study of the hypoelliptic regularity not of the solutions $u$ of the kinetic equation \eqref{11042019E1}, but of their averages
$$\rho_{\phi}(t,x) = \int_{\mathbb R^n}u(t,x,v)\phi(v)\ \mathrm dv,\quad\phi\in C^{\infty}_c(\mathbb R^n).$$
Such a study will not be performed in this work and we refer to the recent work \cite{JLT} where a large state of art concerning this domain is made.
Obtaining hypoelliptic estimates like \eqref{06052020E1} for a large class of transport operators would allow to understand the global phenomena of transmission of regularity illustrated in the particular case \eqref{15072019E2}, where regularity in the space variable $x\in\mathbb R^n$ is provided from regularity in the velocity variable $v\in\mathbb R^n$. It would also provide a better understanding of the explicit gain of regularity stated in the estimate \eqref{15072019E2}.

In order to obtain hypoelliptic estimates like \eqref{06052020E1}, we will use a multiplier method inspired by the work \cite{MR2752935} of K. Pravda-Starov in which the author used this kind of method to obtain sharp subelliptic estimates enjoyed by a large class of accretive quadratic differential operators. The very same method was also used by the same author in the work \cite{MR2993058} to obtain similar estimates for systems of quadratic differential operators. More generally, the multiplier method was widely used in the literature to obtain coercive and subelliptic estimates for operators arising from the theory of kinetic equations or (hypoelliptic) estimates for the solutions of such equations including for examples transport equations with external potentials or various models of Boltzmann equations, see \cite{CHLZ, MR2786222, MR3348825, MR3193940, MR3456819}. The construction implemented in this work is quite technical in the general case, but is presented in Section \ref{particular} for a particular operator in order to identify its main steps.

\subsubsection*{Outline of the work} In Section \ref{results}, we state the main results contained in this work. Section \ref{particular} is devoted to present the main steps of the multiplier method applied to a particular transport operator. The general construction is performed in Section \ref{constructionKalman} and the derivation of the hypoelliptic estimates once this construction is achieved is made in Section \ref{hypoelliptic}. Section \ref{appendix} is an appendix containing the proofs of some technical results. 

\subsubsection*{Notations} The following notations and conventions will be used all over the work:
\begin{enumerate}[label=\textbf{\arabic*.},leftmargin=* ,parsep=2pt,itemsep=0pt,topsep=2pt]
\item The canonical Euclidean scalar product of $\mathbb R^n$ is denoted by $\cdot$ and $\vert\cdot\vert$ stands for the associated canonical Euclidean norm.
\item The Japanese bracket $\langle\cdot\rangle$ is defined for all $\xi\in\mathbb R^n$ by $\langle\xi\rangle = \sqrt{1+\vert\xi\vert^2}$.
\item The inner product of $L^2(\mathbb R^n)$ is defined for all $u,v\in L^2(\mathbb R^n)$ by
$$\langle u,v\rangle_{L^2(\mathbb R^n)} = \int_{\mathbb R^n}u(x)\overline{v(x)}\ \mathrm dx.$$
\item In all the estimates appearing in this work, the $L^2$ scalar product $\langle\cdot,\cdot\rangle_{L^2}$ denotes the scalar product of the space $L^2(\mathbb R^{n+1}_{t,x})$ when the operator at play is $L = \partial_t + Bx\cdot\nabla_x$ and stands for the scalar product of the space $L^2(\mathbb R^n_x)$ in the case where $L = Bx\cdot\nabla_x$. The situation is similar with the associated norm $\Vert\cdot\Vert_{L^2}$, and more generally with any Sobolev norm $\Vert\cdot\Vert_{H^s}$, with $s\geq0$.
\item For all function $u\in\mathscr S(\mathbb R^n)$, the Fourier transform of $u$ is denoted $\widehat u$ and defined by
$$\widehat u(\xi) = \int_{\mathbb R^n}e^{-2i\pi x\cdot\xi}u(x)\ \mathrm dx.$$
With this convention, Plancherel's theorem states that 
$$\forall u\in\mathscr S(\mathbb R^n),\quad \Vert\widehat u\Vert_{L^2(\mathbb R^n)} = \Vert u\Vert_{L^2(\mathbb R^n)}.$$
\item In all this work, the vector space $\mathscr S$ denotes the Schwartz space $\mathscr S(\mathbb R^{n+1}_{t,x})$ when the operator at play is $L = \partial_t + Bx\cdot\nabla_x$ and stands for the Schwartz space $\mathscr S(\mathbb R^n_x)$ in the case where $L = Bx\cdot\nabla_x$.
\item For all continuous function $F:\mathbb R^n\rightarrow\mathbb C$, the notation $F(D_x)$ is used to denote the Fourier multiplier associated with the symbol $F(\xi)$.
\item Given $f,g:\mathbb R^n\rightarrow\mathbb R$ two functions, the notation $f\lesssim g$ means that there exists a positive constant $c>0$ such that $f\le cg$, the constant $c>0$ not depending on exterior parameters. Given moreover $\varepsilon>0$, the notation $f\lesssim_{\varepsilon} g$ means that there exists a positive constant $c_{\varepsilon}>0$ depending on $\varepsilon>0$ such that $f\le c_{\varepsilon}g$.
\end{enumerate}

\section{Mains results}\label{results}

This section is devoted to present the main results contained in this paper. As presented in the introduction, when the Kalman rank condition \eqref{10052018E4} holds, we aim at obtaining hypoelliptic estimates of the form \eqref{06052020E1} for the transport operators $L$ defined in \eqref{05042019E6}. Notice from Lemma \ref{29082018E1} that the matrices $B$ and $Q$ satisfy the Kalman rank condition if and only if there exists a non-negative integer $0\le r\le n-1$ satisfying
\begin{equation}\label{10042019E2}
	\Ker(Q)\cap\Ker(QB^T)\cap\ldots\cap\Ker(Q(B^T)^r) = \{0\}.
\end{equation}
The smallest integer $0\le r\le n-1$ satisfying \eqref{10042019E2} is the one that appears in the estimate \eqref{06052020E2}. It will also play a key role in the futur results of this section. 

In the general case, we will prove a weaker estimate than the one stated in \eqref{06052020E1} and obtain a global hypoelliptic estimate of the form 
\begin{equation}\label{06052020E3}
	\Vert u\Vert_{H^{\lambda_r}}\lesssim\big\Vert\langle QD_x\rangle^{\lambda_0}u\big\Vert_{L^2} + \sum_{j=0}^{r-1}\big\Vert\langle Q(B^T)^jD_x\rangle^{q_j}Lu\big\Vert_{H^{s_j}},
\end{equation}
again with $\lambda_0>0$, $0\le q_0,\ldots, q_{r-1}\le 1$ and $s_0,\ldots,s_{r-1}\geq0$. To that end, we will use a multiplier method, as explained in the previous section. The main technical part of this work consists in constructing a multiplier adapted to the problem we are interested in. In our case, this a smooth real-valued symbol $g\in C^{\infty}(\mathbb R^n,\mathbb R)$ satisfying estimates of the following form for all $\xi\in\mathbb R^n$,
\begin{align}\label{07052020E1}
\begin{split}
	& \vert g(\xi)\vert\lesssim\sum_{j=0}^{r-1}\langle Q(B^T)^j\xi\rangle^{q_j}\langle\xi\rangle^{s_j}, \\
	& \langle\xi\rangle^{\lambda_r}\lesssim\langle Q\xi\rangle^{\lambda_0} + B^T\xi\cdot\nabla_{\xi}g(\xi).
\end{split}
\end{align}
Obtaining the hypoelliptic estimate \eqref{06052020E3} from \eqref{07052020E1} is performed in Section \ref{hypoelliptic}. While doing this construction, the regularity exponent $\lambda_r>0$ involved in \eqref{06052020E3} and \eqref{07052020E1} appears as the last term of the family of positive real numbers $(\lambda_0,\ldots,\lambda_r)$ recursively defined by
\begin{equation}\label{06052020E4}
	\forall j\in\{0,\ldots,r-1\},\quad\lambda_{j+1}\bigg(\frac{1-q_j}{\lambda_j}+1\bigg) = 1+s_j.
\end{equation}
Notice that the real numbers $\lambda_0,\ldots,\lambda_r>0$ defined in \eqref{06052020E5} satisfy the above recurrence relation with $q_j = s_j = 0$. This relation illustrates how the regularity of the source term $Lu$ influences the global Sobolev regularity of the function $u$ through the regularity exponent $\lambda_r>0$ appearing in the hypoelliptic estimate \eqref{06052020E3}. In the case where $r\geq2$, we will need to make the following technical assumption on the family $(\lambda_0,\ldots,\lambda_r)$ constructed just before
\begin{equation}\label{15042020E1}
	\forall j\in\{2,\ldots,r\},\quad \frac{\lambda_{j-1}}{\lambda_{j-2}}+\frac{\lambda_{j-1}}{\lambda_j}\le 2.
\end{equation}
Again, notice that the real numbers $\lambda_0,\ldots,\lambda_r>0$ defined in \eqref{06052020E5} satisfy the estimate \eqref{15042020E1} when $r\geq2$. The main result contained in this paper is the following:

\begin{thm}\label{18032020T1} Let $B$ be a real $n\times n$ matrix and $L$ be one of the two transport operators defined in \eqref{05042019E6}. We consider $Q$ a real symmetric positive semidefinite $n\times n$ matrix. We assume that the matrices $B$ and $Q$ satisfy the Kalman rank condition \eqref{10052018E4} and that the smallest integer $0\le r\le n-1$ satisfying \eqref{10042019E2} is positive. Let $\lambda_0>0$, $0\le q_{r-2}\le q_{r-1}\le1$ and $s_{r-1}\geq s_{r-2}\geq0$ be some non-negative real numbers. By setting $q_j = s_j = 0$ when $j\le r-3$, we consider the family of positive real numbers $(\lambda_0,\ldots,\lambda_r)$ recursively defined by \eqref{06052020E4}. When this family $(\lambda_0,\ldots,\lambda_r)$ is decreasing and satisfies the assumption \eqref{15042020E1} in the case where $r\geq2$, there exists a positive constant $c>0$ such that for all $u\in\mathscr S$,
$$\Vert u\Vert_{H^{\lambda_r}}\le c\Big(\big\Vert\langle QD_x\rangle^{\lambda_0}u\big\Vert_{L^2} + \sum_{j=0}^{r-1}\big\Vert\langle Q(B^T)^jD_x\rangle^{q_j}Lu\big\Vert_{H^{s_j}}\Big).$$
\end{thm}

The exponent $\lambda_r>0$ appearing in the above statement is explicitly given by  
$$\lambda_r = \frac{\lambda_0(1+s_{r-2})(1+s_{r-1})}{(1-q_{r-2})(1-q_{r-1}) + \lambda_0(1+s_{r-2}) + \lambda_0(1-q_{r-1}) + \lambda_0(1-q_{r-2})(1-q_{r-1})(r-2)},$$
see Subsection \ref{regexp} in the Appendix. We also refer to this subsection where the various assumptions made in Theorem \ref{18032020T1} are explicitly expressed in terms of $\lambda_0$ and the $q_j,s_j$. Notice that when $r\geq2$, $0\le q_{r-2} = q_{r-1} = q\le1$  and $s_{r-2} = s_{r-1} = s\geq0$, the decreasing property of the family $(\lambda_0,\ldots,\lambda_r)$ and the assumption \eqref{15042020E1} are equivalent to the following simple condition
$$(q+s)(1+\lambda_0(r-2))<\lambda_0.$$

In the case where $r=1$, the estimate stated in the above theorem writes as
$$\Vert u\Vert_{H^{\lambda_1}}\lesssim\big\Vert\langle QD_x\rangle^{\lambda_0}u\big\Vert_{L^2} + \big\Vert\langle QD_x\rangle^{q_0}Lu\big\Vert_{H^{s_0}},$$
while in the case where $r\geq2$, this estimate is
$$\Vert u\Vert_{H^{\lambda_r}}\lesssim\big\Vert\langle QD_x\rangle^{\lambda_0}u\big\Vert_{L^2} + \big\Vert\langle Q(B^T)^{r-2}D_x\rangle^{q_{r-2}}Lu\big\Vert_{H^{s_{r-2}}}
+ \big\Vert\langle Q(B^T)^{r-1}D_x\rangle^{q_{r-1}}Lu\big\Vert_{H^{s_{r-1}}}.$$
We therefore obtained the hypoelliptic estimate \eqref{06052020E3} in the cases where $r=1$ and $r=2$ under reasonable assumptions and when $r\geq3$ while assuming moreover that $q_j = s_j = 0$ for all $0\le j\le r-3$. We conjecture that this last assumption is only technical and that the estimate \eqref{06052020E3} holds true without it. The conditions $q_{r-2}\le q_{r-1}$ and $s_{r-1}\geq s_{r-2}$ are also technical. Coupled with the fact that $\lambda_{r-1}>\lambda_r$, they allow to get the estimates
$$\lambda_r>q_{r-1}+s_{r-1},\quad\text{and therefore}\quad\lambda_r>q_{r-2}+s_{r-2},$$
as checked in Lemma \ref{10042019L1}, which appears necessary in the proof of Theorem \ref{18032020T1}, but not necessarily natural. The two other assumptions will be discussed later.

Even if we could not obtain anisotropic hypoelliptic estimates in the general case as presented in \eqref{06052020E1}, the multiplier method we are using in this work allows to obtain such estimates (and even sharper) in particular cases as the one presented in the 

\begin{prop}\label{13042020P1} Let $L$ be one of the two following transport operators
$$\partial_t + x_0\cdot\nabla_{x_1}+x_1\cdot\nabla_{x_2},\quad x_0\cdot\nabla_{x_1}+x_1\cdot\nabla_{x_2},\quad t\in\mathbb R,\ x_0,x_1,x_2\in\mathbb R^n.$$
We consider $\lambda_0>0$, $0\le q_0,q_1\le1$ and $s_1,s_2\geq0$ some non-negative real numbers, and $\lambda_1,\lambda_2>0$ the positive real numbers recursively defined by
$$\lambda_1\bigg(\frac{1-q_0}{\lambda_0}+1\bigg) = 1+s_1\quad\text{and}\quad\lambda_2\bigg(\frac{1-q_1}{\lambda_1}+1\bigg) = 1+s_2.$$
When the real numbers $\lambda_0,\lambda_1,\lambda_2>0$ satisfy the condition \eqref{15042020E1}, there exists a positive constant $c>0$ such that for all $u\in\mathscr S$,
\begin{align*}
	& \big\Vert\langle D_{x_1}\rangle^{\lambda_1}u\big\Vert_{L^2}\le c\big(\big\Vert\langle D_{x_0}\rangle^{\lambda_0}u\big\Vert_{L^2}
	+ \big\Vert\langle D_{x_0}\rangle^{q_0}\langle D_{x_1}\rangle^{s_1}Lu\big\Vert_{L^2}\big), \\[2pt]
	& \big\Vert\langle D_{x_2}\rangle^{\lambda_2}u\big\Vert_{L^2}\le c\big(\big\Vert\langle D_{x_1}\rangle^{\lambda_1}u\big\Vert_{L^2}
	+ \big\Vert\langle D_{x_1}\rangle^{q_1}\langle D_{x_2}\rangle^{s_2}Lu\big\Vert_{L^2}\big).
\end{align*}
\end{prop}
 
The main steps of the proof of this proposition are outline in Section \ref{particular} in order to present the multiplier method. Keeping the context of Proposition \ref{13042020P1} and assuming moreover that $\lambda_0>\lambda_1>\lambda_2$, we deduce that for all $u\in\mathscr S$,
$$\Vert u\Vert_{H^{\lambda_2}}\lesssim\big\Vert\langle D_{x_0}\rangle^{\lambda_0}u\big\Vert_{L^2}
+ \big\Vert\langle D_{x_0}\rangle^{q_0}\langle D_{x_1}\rangle^{s_1}Lu\big\Vert_{L^2} + \big\Vert\langle D_{x_1}\rangle^{q_1}\langle D_{x_2}\rangle^{s_2}Lu\big\Vert_{L^2}.$$
This particular example therefore suggests that assuming that the family $(\lambda_0,\ldots,\lambda_r)$ is decreasing in Theorem \ref{18032020T1} is natural. Moreover, as it is present in both results Theorem \ref{18032020T1} and Proposition \ref{13042020P1}, we believe that the assumption \eqref{15042020E1} is necessary to obtain an estimate like \eqref{06052020E1}. This is reinforced by the fact the estimate \eqref{06052020E2} holds true with regularity exponents $\lambda_j>0$ given in \eqref{06052020E5} satisfying this condition, as we have already mentioned.

\medskip

\begin{ex} Let $L$ be one of the two following transport operators
$$\partial_t + v\cdot\nabla_x,\quad v\cdot\nabla_x,\quad t\in\mathbb R,\ x,v\in\mathbb R^n.$$
We also consider $p>0$, $0\le q\le 1$ and $s\geq0$ some non-negative real numbers. The very same proof as the one used to obtain Proposition \ref{13042020P1} allows to derive the following estimate for all $u\in\mathscr S$,
$$\big\Vert\langle D_x\rangle^{\frac{p(1+s)}{1-q+p}} u\big\Vert_{L^2}\lesssim\big\Vert\langle D_v\rangle^pu\big\Vert_{L^2} + \big\Vert\langle D_x\rangle^s\langle D_v\rangle^qLu\big\Vert_{L^2}.$$
Even more, the same strategy of proof also allows to directly obtain that given a source term $f\in L^2_{t,x,v}$ satisfying $\langle D_x\rangle^s\langle D_v\rangle^qf\in L^2_{t,x,v}$, any solution $u\in L^2_{t,x,v}$ of the transport equation \eqref{11042019E1} such that $\langle D_v\rangle^pu\in L^2_{t,x,v}$ satisfies an hypoelliptic estimate of the form 
$$\big\Vert\langle D_x\rangle^{\frac{p(1+s)}{1-q+p}} u\big\Vert_{L^2_{t,x,v}}\lesssim\big\Vert\langle D_v\rangle^pu\big\Vert_{L^2_{t,x,v}} + \big\Vert\langle D_x\rangle^s\langle D_v\rangle^qf\big\Vert_{L^2_{t,x,v}}.$$
This this is a generalization of the estimate \eqref{15072019E2}.
\end{ex}

\medskip

Notice that the two estimates stated in the above proposition are sharper versions of the hypoelliptic estimate \eqref{06052020E1}. They suggest that more general estimates of the form 
$$\sum_{j=1}^r\big\Vert\langle Q(B^T)^jD_x\rangle^{\lambda_j}u\big\Vert_{L^2}\lesssim\big\Vert\langle QD_x\rangle^{\lambda_0}u\big\Vert_{L^2} + \sum_{j=0}^{r-1}\big\Vert\langle Q(B^T)^jD_x\rangle^{q_j}\langle Q(B^T)^{j+1}D_x\rangle^{s_{j+1}}Lu\big\Vert_{L^2},$$
might hold true, the regularity exponents $\lambda_1,\ldots,\lambda_r>0$ being recursively defined by \eqref{06052020E4}, the matrices $B$ and $Q$ satisfying the Kalman rank condition \eqref{10052018E4} or not. Nevertheless, this kind of estimate seems out of reach for the moment, even under the Kalman rank condition. However, there are good reasons to believe that the weaker estimate \eqref{06052020E1} can be obtained with a multiplier method as the one used in this work.

\section{The multiplier method in a particular case}\label{particular}

The aim of this subsection is to present the main steps of the proof of Proposition \ref{13042020P1}. No technical calculus will be detailed here, the objective is to illustrate how the multiplier method works on a particular example. Moreover, all the calculus that are omitted in this section are of the same type as those detailed in Section \ref{constructionKalman} where a multiplier is constructed in the general case. Let $L$ be the following transport operator, with $t\in\mathbb R$ and $(x_0,x_1,x_2)\in\mathbb R^{3n}$,
$$L = \partial_t + x_0\cdot\nabla_{x_1}+x_1\cdot\nabla_{x_2}.$$
The following proof works exactly the same when $L$ stands for the following autonomous transport operator
\begin{equation}\label{03052020E1}
	x_0\cdot\nabla_{x_1}+x_1\cdot\nabla_{x_2},
\end{equation}
as we will see. We also consider $\lambda_0>0$, $0\le q_0,q_1\le1$ and $s_1,s_2\geq0$ some non-negative real numbers, and $\lambda_1,\lambda_2>0$ the positive real numbers recursively defined by
\begin{equation}\label{14042020E2}
	\lambda_1\bigg(\frac{1-q_0}{\lambda_0}+1\bigg) = 1+s_1\quad\text{and}\quad\lambda_2\bigg(\frac{1-q_1}{\lambda_1}+1\bigg) = 1+s_2.
\end{equation}
Let us assume that the real numbers $\lambda_0,\lambda_1,\lambda_2>0$ satisfy the condition \eqref{15042020E1}, that is,
\begin{equation}\label{24062020E1}
	\frac{\lambda_1}{\lambda_0}+\frac{\lambda_1}{\lambda_2}\le2.
\end{equation}
The result we aim at proving states that the following estimates hold for all $u\in\mathscr S$,
\begin{align}
	& \big\Vert\langle D_{x_2}\rangle^{\lambda_2}u\big\Vert_{L^2}\lesssim\big\Vert\langle D_{x_1}\rangle^{\lambda_1}u\big\Vert_{L^2}
	+ \big\Vert\langle D_{x_1}\rangle^{q_1}\langle D_{x_2}\rangle^{s_2}Lu\big\Vert_{L^2}, \label{28042020E1} \\[2pt]
	& \big\Vert\langle D_{x_1}\rangle^{\lambda_1}u\big\Vert_{L^2}\lesssim\big\Vert\langle D_{x_0}\rangle^{\lambda_0}u\big\Vert_{L^2}
	+ \big\Vert\langle D_{x_0}\rangle^{q_0}\langle D_{x_1}\rangle^{s_1}Lu\big\Vert_{L^2}. \label{03052020E2}
\end{align}
The main step consists in constructing smooth real-valued symbols $g_1\in C^{\infty}(\mathbb R^{2n},\mathbb R^n)$ and $g_2\in C^{\infty}(\mathbb R^{3n},\mathbb R)$ satisfying the following estimates for all $(\xi_0,\xi_1,\xi_2)\in\mathbb R^{3n}$, 
\begin{align}
	& \vert g_1(\xi_1,\xi_2)\vert\lesssim\langle\xi_1\rangle^{q_1}\langle\xi_2\rangle^{s_2}\langle\xi_2\rangle^{\lambda_2}, \label{30042020E1} \\
	& \vert g_2(\xi_0,\xi_1,\xi_2)\vert\lesssim\langle\xi_0\rangle^{q_0}\langle\xi_1\rangle^{s_1}\langle\xi_1\rangle^{\lambda_1}, \label{30042020E2}
\end{align}
and such that
\begin{align}
	& \langle\xi_2\rangle^{2\lambda_2}\lesssim\langle\xi_1\rangle^{2\lambda_1} + \xi_2\cdot\nabla_{\xi_1}g_1(\xi_1,\xi_2), \label{30042020E3} \\
	& \langle\xi_1\rangle^{2\lambda_1} \lesssim\langle\xi_0\rangle^{2\lambda_0} + \xi_2\cdot\nabla_{\xi_1}g_1(\xi_1,\xi_2) + \xi_2\cdot\nabla_{\xi_1}g_2(\xi_0,\xi_1,\xi_2) + \xi_1\cdot\nabla_{\xi_0} g_2(\xi_0,\xi_1,\xi_2). \label{30042020E4}
\end{align}
Once these constructions are performed, the estimates \eqref{28042020E1} can be obtained in the following way. Setting $x = (x_0,x_1,x_2)\in\mathbb R^{3n}$ the global space variable, we begin by considering the following quantity for all $u\in\mathscr S$,
$$\Reelle\big\langle Lu,g_1(D_x)u\big\rangle_{L^2} = \Reelle\big\langle\partial_tu,g_1(D_x)u\big\rangle_{L^2} + \Reelle\big\langle(x_0\cdot\nabla_{x_1}+x_1\cdot\nabla_{x_2})u,g_1(D_x)u\big\rangle_{L^2}.$$
On the one hand, since the operator $\partial_t$ is formally skew-selfadjoint, $g(D_x)$ is formally selfadjoint (the symbol $g$ being real-valued) and that these two operators commute, we get that for all $u\in\mathscr S$,
\begin{multline*}
	\Reelle\big\langle\partial_tu,g_1(D_x)u\big\rangle_{L^2} = - \Reelle\big\langle g_1(D_x)u,\partial_tu\big\rangle_{L^2} \\
	= -\Reelle\overline{\big\langle\partial_tu,g_1(D_x)u\big\rangle}_{L^2} = -\Reelle\big\langle\partial_tu,g_1(D_x)u\big\rangle_{L^2},
\end{multline*}
and as a consequence,
$$\Reelle\big\langle\partial_tu,g_1(D_x)u\big\rangle_{L^2} = 0.$$
Notice that this equality is the reason why the proof we are presenting now works the same when $L$ stands for the autonomous operator \eqref{03052020E1}. On the other hand, by setting $\xi\in\mathbb R^{3n}$ the dual variable of $x\in\mathbb R^{3n}$, we deduce from Plancherel's theorem that for all $u\in\mathscr S$,
$$\Reelle\big\langle(x_0\cdot\nabla_{x_1}+x_1\cdot\nabla_{x_2})u,g_1(D_x)u\big\rangle_{L^2} = 	
\Reelle\big\langle(\xi_2\cdot\nabla_{\xi_1}+\xi_1\cdot\nabla_{\xi_0})\widehat u(\xi),g_1(\xi)\widehat u(\xi)\big\rangle_{L^2},$$
where $\widehat u$ denotes the Fourier transform of the Schwartz function $u\in\mathscr S$ with respect to the space variable $x\in\mathbb R^{3n}$. Using the fact that the multiplier $g_1$ does not depend on the variable $\xi_0\in\mathbb R^n$, it therefore follows from Leibniz' formula that for all $u\in\mathscr S$,
\begin{multline*}
	\Reelle\big\langle(x_0\cdot\nabla_{x_1}+x_1\cdot\nabla_{x_2})u,g_1(D_x)u\big\rangle_{L^2}
	= -\Reelle\big\langle\widehat u(\xi),(\xi_2\cdot\nabla_{\xi_1}g_1(\xi))\widehat u(\xi)\big\rangle_{L^2} \\[5pt]
	- \Reelle\big\langle\widehat u(\xi),g_1(\xi)(\xi_2\cdot\nabla_{\xi_1}+\xi_1\cdot\nabla_{\xi_0})\widehat u(\xi)\big\rangle_{L^2},
\end{multline*}
with
$$\Reelle\big\langle\widehat u(\xi),g_1(\xi)(\xi_2\cdot\nabla_{\xi_1}+\xi_1\cdot\nabla_{\xi_0})\widehat u(\xi)\big\rangle_{L^2} 
= \Reelle\big\langle(x_0\cdot\nabla_{x_1}+x_1\cdot\nabla_{x_2})u,g_1(D_x)u\big\rangle_{L^2}.$$
Thus, we deduce that for all $u\in\mathscr S$,
$$-2\Reelle\big\langle Lu,g_1(D_x)u\big\rangle_{L^2} = \Reelle\big\langle\widehat u(\xi),(\xi_2\cdot\nabla_{\xi_1}g_1(\xi))\widehat u(\xi)\big\rangle_{L^2}.$$
Plancherel's theorem and \eqref{30042020E3} then imply that for all $u\in\mathscr S$,
$$\big\Vert\langle D_{x_2}\rangle^{\lambda_2}u\big\Vert^2_{L^2}\lesssim\big\Vert\langle D_{x_1}\rangle^{\lambda_1}u\big\Vert^2_{L^2}-2\Reelle\big\langle Lu,g_1(D_x)u\big\rangle_{L^2}.$$
Finally, the estimate \eqref{30042020E1} and the Cauchy-Schwarz inequality show that for all $u\in\mathscr S$,
$$\big\vert\Reelle\big\langle Lu,g_1(D_x)u\big\rangle_{L^2}\big\vert\lesssim\big\Vert\langle D_{x_1}\rangle^{q_1}\langle D_{x_2}\rangle^{s_2}Lu\big\Vert_{L^2}\big\Vert\langle D_{x_2}\rangle^{\lambda_2}u\big\Vert_{L^2}.$$
This ends the proof of the estimate \eqref{28042020E1}, after using Young's inequality. By considering the symbol $g = g_1+g_2\in C^{\infty}(\mathbb R^{3n},\mathbb R)$ instead of the symbol $g_1$, we obtain in the very same way that for all $u\in\mathscr S$,
\begin{multline*}
	\big\Vert\langle D_{x_1}\rangle^{\lambda_1}u\big\Vert^2_{L^2}\lesssim\big\Vert\langle D_{x_0}\rangle^{\lambda_0}u\big\Vert^2_{L^2} 
	+ \big\Vert\langle D_{x_0}\rangle^{q_0}\langle D_{x_1}\rangle^{s_1}Lu\big\Vert_{L^2}\big\Vert\langle D_{x_1}\rangle^{\lambda_1}u\big\Vert_{L^2} \\[5pt]
	+ \big\Vert\langle D_{x_1}\rangle^{q_1}\langle D_{x_2}\rangle^{s_2}Lu\big\Vert_{L^2}\big\Vert\langle D_{x_2}\rangle^{\lambda_2}u\big\Vert_{L^2}.
\end{multline*}
We then obtain the estimate \eqref{03052020E2} from \eqref{28042020E1}, another use of Young's inequality and choosing $q_1=s_2=0$. Now, let us explain how to construct such symbols $g_1$ and $g_2$. Considering a $C^{\infty}_0(\mathbb R,[0,1])$ cut-off function $\psi$ localized in a neighborhood of the origin, we begin by considering the symbol $g_1\in C^{\infty}(\mathbb R^{2n})$ defined for all $(\xi_1,\xi_2)\in\mathbb R^{2n}$ by
$$g_1(\xi_1,\xi_2) = \frac{\xi_1\cdot\xi_2}{\langle\xi_2\rangle^{2-2\lambda_2}}\ \psi\bigg(\frac{\vert\xi_1\vert^2}{\langle\xi_2\rangle^{2\lambda_2/\lambda_1}}\bigg).$$
Notice that the symbol $g_1$ satisfies the estimate \eqref{30042020E1}, since we deduce from \eqref{14042020E2} and the fact that the function $\psi$ is supported near the origin that for all $(\xi_1,\xi_2)\in\mathbb R^{2n}$,
$$\langle\xi_1\rangle^{-q_1}\vert g(\xi_1,\xi_2)\vert\le\frac{\langle\xi_1\rangle^{1-q_1}\langle\xi_2\rangle}{\langle\xi_2\rangle^{2-2\lambda_2}}
\bigg\vert\psi\bigg(\frac{\vert\xi_1\vert^2}{\langle\xi_2\rangle^{2\lambda_2/\lambda_1}}\bigg)\bigg\vert
\lesssim\frac{\langle\xi_2\rangle^{(1-q_1)\lambda_2/\lambda_1+1}}{\langle\xi_2\rangle^{2-2\lambda_2}}
= \langle\xi_2\rangle^{s_2}\langle\xi_2\rangle^{\lambda_2}.$$
Moreover, computing $\nabla_{\xi_1} g_1$ and using the support localization of the function $\psi$ anew, we deduce that for all $(\xi_1,\xi_2)\in\mathbb R^{2n}$,
\begin{equation}\label{30042020E5}
	\langle\xi_2\rangle^{2\lambda_2}\lesssim\vert\xi_1\vert^{2\lambda_1}w\bigg(\frac{\vert\xi_1\vert^2}{\langle\xi_2\rangle^{2\lambda_2/\lambda_1}}\bigg) + \xi_2\cdot\nabla_{\xi_1} g_1(\xi_1,\xi_2) + 1,
\end{equation}
where $w$ is a $C^{\infty}(\mathbb R,[0,1])$ function vanishing in a neighborhood of the origin and constant equal to $1$ at infinity, controlling the functions $1-\psi$ and $\vert\psi'\vert$. Since the function $w$ is bounded, we therefore obtain the estimate \eqref{30042020E3}. Moreover, since $1 = (1-w) + w$ and that the function $1-w$ is localized near the origin, we obtain that for all $(\xi_1,\xi_2)\in\mathbb R^{2n}$,
\begin{equation}\label{30042020E6}
	\vert\xi_1\vert^{2\lambda_1}\lesssim\langle\xi_2\rangle^{2\lambda_2} + \vert\xi_1\vert^{2\lambda_1}w\bigg(\frac{\vert\xi_1\vert^2}{\langle\xi_2\rangle^{2\lambda_2/\lambda_1}}\bigg).
\end{equation}
Therefore, to derive \eqref{30042020E4}, we need to control the term 
$$\vert\xi_1\vert^{2\lambda_1}w\bigg(\frac{\vert\xi_1\vert^2}{\langle\xi_2\rangle^{2\lambda_2/\lambda_1}}\bigg).$$
To that end, we consider $\Gamma\gg1$ a large positive constant whose value will be chosen later. By using that $1\lesssim\psi+w$, we deduce that for all $(\xi_1,\xi_2)\in\mathbb R^{2n}$,
\begin{multline}\label{30042020E7}
	\vert\xi_1\vert^{2\lambda_1}w\bigg(\frac{\vert\xi_1\vert^2}{\langle\xi_2\rangle^{2\lambda_2/\lambda_1}}\bigg) \lesssim
	\vert\xi_1\vert^{2\lambda_1}w\bigg(\frac{\vert\xi_1\vert^2}{\langle\xi_2\rangle^{2\lambda_2/\lambda_1}}\bigg)w\bigg(\frac{\Gamma^2\vert\xi_0\vert^2}{\vert\xi_1\vert^{2\lambda_1/\lambda_0}}\bigg) \\
	+\vert\xi_1\vert^{2\lambda_1}w\bigg(\frac{\vert\xi_1\vert^2}{\langle\xi_2\rangle^{2\lambda_2/\lambda_1}}\bigg)\psi\bigg(\frac{\Gamma^2\vert\xi_0\vert^2}{\vert\xi_1\vert^{2\lambda_1/\lambda_0}}\bigg),
\end{multline}
and the support localization of $w$ implies that
\begin{equation}\label{30042020E8}
	\vert\xi_1\vert^{2\lambda_1}w\bigg(\frac{\vert\xi_1\vert^2}{\langle\xi_2\rangle^{2\lambda_2/\lambda_1}}\bigg)w\bigg(\frac{\Gamma^2\vert\xi_0\vert^2}{\vert\xi_1\vert^{2\lambda_1/\lambda_0}}\bigg)\lesssim\Gamma^{2\lambda_0}\vert\xi_0\vert^{2\lambda_0}.
\end{equation}
In order to control the term
$$\vert\xi_1\vert^{2\lambda_1}w\bigg(\frac{\vert\xi_1\vert^2}{\langle\xi_2\rangle^{2\lambda_2/\lambda_1}}\bigg)\psi\bigg(\frac{\Gamma^2\vert\xi_0\vert^2}{\vert\xi_1\vert^{2\lambda_1/\lambda_0}}\bigg),$$
we consider the real-valued symbol $g_2\in C^{\infty}(\mathbb R^{3n},\mathbb R)$ defined for all $(\xi_0,\xi_1,\xi_2)\in\mathbb R^{3n}$ by
$$g_2(\xi_0,\xi_1,\xi_2) = \frac{\xi_0\cdot\xi_1}{\vert\xi_1\vert^{2-2\lambda_1}}\ w\bigg(\frac{\vert\xi_1\vert^2}{\langle\xi_2\rangle^{2\lambda_2/\lambda_1}}\bigg)\psi\bigg(\frac{\Gamma^2\vert\xi_0\vert^2}{\vert\xi_1\vert^{2\lambda_1/\lambda_0}}\bigg).$$
Notice that the symbol $g_2$ is well-defined thanks to the fact that
$$\langle\xi_2\rangle^{\lambda_2}\lesssim\vert\xi_1\vert^{\lambda_1}\quad\text{on}\quad\Supp w\bigg(\frac{\vert\xi_1\vert^2}{\langle\xi_2\rangle^{2\lambda_2/\lambda_1}}\bigg).$$
Moreover, the same arguments as the ones used to justify that the symbol $g_1$ satisfies the estimate \eqref{30042020E1} also allow to prove that the estimate \eqref{30042020E2} holds for the symbol $g_2$ we have just introduced. Until now, we have not needed to use the technical assumption \eqref{24062020E1}. Exploiting this relation on $\lambda_0,\lambda_1,\lambda_2$ and the support localizations of the functions $\psi$ and $w$ as before, a direct calculus implies that for all $(\xi_0,\xi_1,\xi_2)\in\mathbb R^{3n}$,
\begin{equation}\label{30042020E9}
	\vert\xi_1\vert^{2\lambda_1}w\bigg(\frac{\vert\xi_1\vert^2}{\langle\xi_2\rangle^{2\lambda_2/\lambda_1}}\bigg)\psi\bigg(\frac{\Gamma^2\vert\xi_0\vert^2}{\vert\xi_1\vert^{2\lambda_1/\lambda_0}}\bigg)\lesssim\Gamma^{2\lambda_0}\vert\xi_0\vert^{2\lambda_0} + \xi_1\cdot\nabla_{\xi_0}g_2(\xi_0,\xi_1,\xi_2),
\end{equation}
and
\begin{equation}\label{30042020E10}
	-\frac1{\Gamma}\vert\xi_1\vert^{2\lambda_1}\lesssim\xi_2\cdot\nabla_{\xi_1}g_2(\xi_0,\xi_1,\xi_2).
\end{equation}
Gathering the estimates \eqref{30042020E6}, \eqref{30042020E7}, \eqref{30042020E8}, \eqref{30042020E9}, \eqref{30042020E10}, we obtain that there exists a positive constant $c>0$ independent of $\Gamma\gg1$ such that for all $(\xi_0,\xi_1,\xi_2)\in\mathbb R^{3n}$,
\begin{multline*}
	\bigg(1-\frac c{\Gamma}\bigg)\vert\xi_1\vert^{2\lambda_1}\lesssim_{\ \Gamma}\langle\xi_0\rangle^{2\lambda_0} + \xi_2\cdot\nabla_{\xi_1}g_1(\xi_1,\xi_2) \\
	+ \xi_2\cdot\nabla_{\xi_1}g_2(\xi_0,\xi_1,\xi_2) + \xi_1\cdot\nabla_{\xi_0} g_2(\xi_0,\xi_1,\xi_2).
\end{multline*}
Taking $\Gamma\gg1$ large enough therefore ends the proof of \eqref{30042020E4}.

\section{Hypoelliptic estimates \textit{via} the multiplier method in the general case}\label{hypoelliptic}

This subsection is devoted to the proof of Theorem \ref{18032020T1}. We consider $B$ a real $n\times n$ matrix and $L$ one of the two transport operators defined in \eqref{05042019E6}. We also consider $Q$ a real symmetric positive semidefinite $n\times n$ matrix. Let us assume that the matrices $B$ and $Q$ satisfy the Kalman rank condition \eqref{10052018E4}, the associated Kalman index defined in \eqref{10042019E2} satisfying $r\geq1$. We consider $\lambda_0>0$, $0\le q_{r-2}\le q_{r-1}\le1$ and $s_{r-1}\geq s_{r-2}\geq0$ some non-negative real numbers. By setting $q_j = s_j = 0$ when $j\le r-3$, we define recursively the positive real numbers $\lambda_1,\ldots,\lambda_r>0$ by
$$\forall j\in\{0,\ldots,r-1\},\quad\lambda_{j+1}\bigg(\frac{1-q_j}{\lambda_j}+1\bigg) = 1+s_j.$$
Let us assume that the family $(\lambda_0,\ldots,\lambda_r)$ is decreasing and satisfies the assumption \eqref{15042020E1}. We aim at establishing that there exists a positive constant $c>0$ such that for all Schwartz function $u\in\mathscr S$,
\begin{equation}\label{31032020E6}
	\Vert u\Vert_{H^{\lambda_r}}\le c\Big(\big\Vert\langle QD_x\rangle^{\lambda_0}u\big\Vert_{L^2} + \sum_{j=0}^{r-1}\big\Vert\langle Q(B^T)^jD_x\rangle^{q_j}Lu\big\Vert_{H^{s_j}}\Big).
\end{equation}
The proof of the estimate \eqref{31032020E6} is based on a multiplier method. Since the matrices $B$ and $Q$ satisfy the Kalman rank condition \eqref{10052018E4}, the definition \eqref{10042019E2} of the associated integer $r\geq1$ implies that there exists a positive constant $c_0>0$ such that
$$\forall\xi\in\mathbb R^n,\quad \sum_{j=0}^r\big\vert Q(B^T)^j\xi\big\vert^2\geq c_0\vert\xi\vert^2.$$
It therefore follows from Proposition \ref{05122018P1} that there exists a real-valued symbol $g\in C^{\infty}(\mathbb R^n)$ satisfying that there exist some positive constants $c_1,c_2,c_3>0$ such the estimates
\begin{equation}\label{04122018E4}
	\vert g(\xi)\vert\le c_1\sum_{j=0}^{r-1}\langle Q(B^T)^j\xi\rangle^{q_j}\langle\xi\rangle^{s_j},
\end{equation}
and
\begin{equation}\label{04122018E5}
	\langle\xi\rangle^{\lambda_r}\le c_2\langle Q\xi\rangle^{\lambda_0} + c_3B^T\xi\cdot\nabla_{\xi}g(\xi),
\end{equation}
hold for all $\xi\in\mathbb R^n$. Notice that the operators $g(D_x)$ and $\partial_t$ are respectively formally selfadjoint and skew-selfadjoint, because the symbol $g$ is real-valued, and commute. Moreover, the operator $\Op^w(Bx\cdot i\xi)$ defined by the Weyl quantization of the symbol $Bx\cdot i\xi$ is formally skew-selfadjoint, because this symbol is purely imaginary. We therefore deduce from the definition \eqref{05042019E6} of the operator $L$ and Lemma \ref{24042020L1} that for all $u\in\mathscr S$,
\begin{equation}\label{04122018E6}
	\Reelle\big\langle Lu,g(D_x)u\big\rangle_{L^2}
	= \frac12\big\langle\big[g(D_x),\Op^w(Bx\cdot i\xi)\big]u,u\big\rangle_{L^2}
	-\frac12\Tr(B)\big\langle u,g(D_x)u\big\rangle_{L^2},
\end{equation}
where we used that 
\begin{equation}\label{28062019E2}
	Bx\cdot\nabla_x = \Op^w(Bx\cdot i\xi) - \frac12\Tr(B).
\end{equation}
Moreover, the symbol $Bx\cdot i\xi$ is a first order polynomial and elements of Weyl calculus, see e.g. \cite{MR2304165} (Theorem 18.5.6), show that the commutator between the operators $g(D_x)$ and $\Op^w(Bx\cdot i\xi)$ is given by
\begin{equation}\label{04122018E7}
	\big[g(D_x),\Op^w(Bx\cdot i\xi)\big] 
	= \Op^w\bigg(\frac1i\big\{g(\xi),Bx\cdot i\xi\big\}\bigg) = \Op^w\big(B^T\xi\cdot\nabla_{\xi}g(\xi)\big).
\end{equation}
We deduce from \eqref{04122018E4}, \eqref{04122018E6}, \eqref{04122018E7} and Plancherel's theorem that for all $u\in\mathscr S$,
\begin{multline}\label{04122018E9}
	\big\langle(B^T\xi\cdot\nabla_{\xi}g(\xi))\mathscr F_x u,\mathscr F_x u\big\rangle_{L^2}
	\le2\Reelle\big\langle Lu,g(D_x)u\big\rangle_{L^2} \\
	+\Tr(B)c_1\sum_{j=0}^{r-1}\big\Vert\langle Q(B^T)^jD_x\rangle^{\frac{q_j}2}\langle D_x\rangle^{\frac{s_j}2}u\big\Vert^2_{L^2},
\end{multline}
where $\mathscr F_x u$ denotes the partial Fourier transform of the Schwartz function $u\in\mathscr S$ with respect to the space variable $x\in\mathbb R^n$. The above estimate implies that for all $u\in\mathscr S$,
\begin{multline}\label{04122018E10}
	\big\langle (c_2\langle Q\xi\rangle^{\lambda_0} + c_3(B^T\xi\cdot\nabla_{\xi}g(\xi)))\mathscr F_x u,\mathscr F_x u\big\rangle_{L^2}
	\le2c_3 \Reelle\big\langle Lu,g(D_x)u\big\rangle_{L^2} \\[5pt]
	+ c_2\big\Vert\langle QD_x\rangle^{\frac{\lambda_0}2}u\big\Vert^2_{L^2}
	+\Tr(B)c_1c_3\sum_{j=0}^{r-1}\big\Vert\langle Q(B^T)^jD_x\rangle^{\frac{q_j}2}\langle D_x\rangle^{\frac{s_j}2}u\big\Vert^2_{L^2}.
\end{multline}
As a consequence of \eqref{04122018E5}, \eqref{04122018E10} and Plancherel's theorem, we get that for all $u\in\mathscr S$,
\begin{multline}\label{04122018E11}
	\big\Vert\langle D_x\rangle^{\frac{\lambda_r}2}u\big\Vert^2_{L^2}
	\le 2c_3\Reelle\big\langle Lu,g(D_x)u\big\rangle_{L^2}
	+ c_2\big\Vert\langle QD_x\rangle^{\frac{\lambda_0}2}u\big\Vert^2_{L^2} \\
	+\Tr(B)c_1c_3\sum_{j=0}^{r-1}\big\Vert\langle Q(B^T)^jD_x\rangle^{\frac{q_j}2}\langle D_x\rangle^{\frac{s_j}2}u\big\Vert^2_{L^2}.
\end{multline}
By applying this estimate to the Schwartz function $\langle D_x\rangle^{\frac{\lambda_r}2}u$, we obtain that for all $u\in\mathscr S$,
\begin{multline}\label{28062019E3}
	\big\Vert\langle D_x\rangle^{\lambda_r}u\big\Vert^2_{L^2}
	\le2c_3\Reelle\big\langle L\langle D_x\rangle^{\frac{\lambda_r}2}u,g(D_x)\langle D_x\rangle^{\frac{\lambda_r}2}u\big\rangle_{L^2} \\[5pt]
	+c_2\big\Vert\langle QD_x\rangle^{\frac{\lambda_0}2}\langle D_x\rangle^{\frac{\lambda_r}2}u\big\Vert^2_{L^2}
	+\Tr(B)c_1c_3\sum_{j=0}^{r-1}\big\Vert\langle Q(B^T)^jD_x\rangle^{\frac{q_j}2}\langle D_x\rangle^{\frac{s_j}2}\langle D_x\rangle^{\frac{\lambda_r}2}u\big\Vert^2_{L^2}.
\end{multline}
We need to control the three quantities appearing in the right-hand side of the above estimate. On the one hand, we deduce from Cauchy-Schwarz' inequality that the second and the third one can be controlled in the following way for all $u\in\mathscr S$,
\begin{equation}\label{28062019E4}
	\big\Vert\langle QD_x\rangle^{\frac{\lambda_0}2}\langle D_x\rangle^{\frac{\lambda_r}2}u\big\Vert^2_{L^2}
	\le\big\Vert\langle QD_x\rangle^{\lambda_0}u\big\Vert_{L^2}\Vert u\Vert_{H^{\lambda_r}},
\end{equation}
and for all $0\le j\le r-1$,
\begin{equation}\label{22072019E1}
	\big\Vert\langle Q(B^T)^jD_x\rangle^{\frac{q_j}2}\langle D_x\rangle^{\frac{s_j}2}\langle D_x\rangle^{\frac{\lambda_r}2}u\big\Vert^2_{L^2}\le\big\Vert\langle Q(B^T)^jD_x\rangle^{q_j}u\big\Vert_{H^{s_j}}\Vert u\Vert_{H^{\lambda_r}}.
\end{equation}
On the other hand, we split the other term in two parts for all $u\in\mathscr S$, one of them involving the commutator between the operators $L$ and $\langle D_x\rangle^{\frac{\lambda_r}2}$,
\begin{multline}\label{28062019E6}
	\Reelle\big\langle L\langle D_x\rangle^{\frac{\lambda_r}2}u,g(D_x)\langle D_x\rangle^{\frac{\lambda_r}2}u\big\rangle_{L^2} \\[5pt]
	= \Reelle\big\langle Lu,g(D_x)\langle D_x\rangle^{\lambda_r}u\big\rangle_{L^2}
	+\Reelle\big\langle\big[L,\langle D_x\rangle^{\frac{\lambda_r}2}\big]u,g(D_x)\langle D_x\rangle^{\frac{\lambda_r}2}u\big\rangle_{L^2}.
\end{multline}
We consider each new term one by one. First, we deduce from Cauchy-Schwarz' inequality and \eqref{04122018E4} that the first one is controlled in the following way for all $u\in\mathscr S$,
\begin{align}\label{16072019E1}
	\Reelle\big\langle Lu,g(D_x)\langle D_x\rangle^{\lambda_r}u\big\rangle_{L^2}
	& = \Reelle\big\langle g(D_x)Lu,\langle D_x\rangle^{\lambda_r}u\big\rangle_{L^2} \\
	& \le c_1\sum_{j=0}^{r-1}\big\Vert\langle Q(B^T)^jD_x\rangle^{q_j}Lu\big\Vert_{H^{s_j}}\Vert u\Vert_{H^{\lambda_r}}. \nonumber
\end{align}
We then write the second term in the following way for all $u\in\mathscr S$,
$$\Reelle\big\langle\big[L,\langle D_x\rangle^{\frac{\lambda_r}2}\big]u,g(D_x)\langle D_x\rangle^{\frac{\lambda_r}2}u\big\rangle_{L^2}
=\Reelle\big\langle g(D_x)\langle D_x\rangle^{-\frac{\lambda_r}2}\big[L,\langle D_x\rangle^{\frac{\lambda_r}2}\big]u,\langle D_x\rangle^{\lambda_r}u\big\rangle_{L^2}.$$
Elements of Weyl calculus, for which we refer anew to \cite{MR2304165} (Theorem 18.5.6), and the definition \eqref{05042019E6} of the operator $L$ imply that the Weyl symbol of the commutator in the above estimate is given by the following Poisson bracket
$$\frac1i\big\{Bx\cdot i\xi,\langle\xi\rangle^{\frac{\lambda_r}2}\big\} = -B^T\xi\cdot\nabla_{\xi}\langle\xi\rangle^{\frac{\lambda_r}2}
=-\frac{\lambda_r}2\langle\xi\rangle^{\frac{\lambda_r}2-2}(B^T\xi\cdot\xi).$$
As a consequence, the operator
$$\langle D_x\rangle^{-\frac{\lambda_r}2}\big[L,\langle D_x\rangle^{\frac{\lambda_r}2}\big],$$
is a Fourier multiplier associated with the following bounded symbol
$$-\frac{\lambda_r}2\langle\xi\rangle^{-\frac{\lambda_r}2}\langle\xi\rangle^{\frac{\lambda_r}2-2}(B^T\xi\cdot\xi)
=-\frac{\lambda_r}2\langle\xi\rangle^{-2}(B^T\xi\cdot\xi)\in L^{\infty}(\mathbb R^n).$$
Cauchy-Schwarz' inequality and the estimate \eqref{04122018E4} therefore imply that for all $u\in\mathscr S$,
\begin{equation}\label{16072019E2}
	\Reelle\big\langle\big[L,\langle D_x\rangle^{\frac{\lambda_r}2}\big]u,g(D_x)\langle D_x\rangle^{\frac{\lambda_r}2}u\big\rangle_{L^2}
	\le\frac{c_1\lambda_r}2\sum_{j=0}^{r-1}\big\Vert\langle Q(B^T)^jD_x\rangle^{q_j}Lu\big\Vert_{H^{s_j}}\Vert u\Vert_{H^{\lambda_r}}.
\end{equation}
Gathering the estimates \eqref{28062019E3}, \eqref{28062019E4}, \eqref{22072019E1}, \eqref{28062019E6}, \eqref{16072019E1} and \eqref{16072019E2}, we deduce that there exists a positive constant $c_4>0$ such that for all $u\in\mathscr S$,
\begin{multline}\label{31032020E8}
	\Vert u\Vert^2_{H^{\lambda_r}}\le c_4\Big(
	\big\Vert\langle QD_x\rangle^{\lambda_0}u\big\Vert_{L^2}\Vert u\Vert_{H^{\lambda_r}}
	+ \sum_{j=0}^{r-1}\big\Vert\langle Q(B^T)^jD_x\rangle^{q_j}u\big\Vert_{H^{s_j}}\Vert u\Vert_{H^{\lambda_r}} \\
	+ \sum_{j=0}^{r-1}\big\Vert\langle Q(B^T)^jD_x\rangle^{q_j}Lu\big\Vert_{H^{s_j}}\Vert u\Vert_{H^{\lambda_r}}\Big).
\end{multline}
Finally, since $\lambda_r>q_j+s_j$ from Lemma \ref{10042019L1}, the inequality $\lambda_{r-1}>\lambda_r$ being assumed, Young's inequality and Plancherel's theorem imply that for all $\varepsilon>0$, there exists a positive constant $c_{\varepsilon}>0$ such that for all $0\le j\le r-1$ and $u\in\mathscr S$,
\begin{equation}\label{31032020E7}
	\big\Vert\langle Q(B^T)^jD_x\rangle^{q_j}u\big\Vert_{H^{s_j}}\lesssim\Vert u\Vert_{H^{q_j+s_j}}
	\lesssim\varepsilon\Vert u\Vert_{H^{\lambda_r}} + c_{\varepsilon}\Vert u\Vert_{L^2}.
\end{equation}
Adjusting the value of $\varepsilon>0$, we derive the estimate \eqref{31032020E6} from \eqref{31032020E8} and \eqref{31032020E7}.
This ends the proof of Theorem \ref{18032020T1}.

\section{Construction of the multiplier in the general case}\label{constructionKalman}

The aim of this section is to construct the real-valued multiplier $g\in C^{\infty}(\mathbb R^n)$ we have used in Section \ref{hypoelliptic} to prove Theorem \ref{18032020T1}. This construction, inspired by the work \cite{MR2752935} (Section 4) in which the author obtained global subelliptic estimates for a large class of accretive quadratic differential operators, is performed in the following proposition:

\begin{prop}\label{05122018P1} Let $B$ and $Q$ be real $n\times n$ matrices, with $Q$ symmetric positive semidefinite. We assume that there exist a positive integer $r\geq1$ and an open subset $\Omega_0\subset\mathbb R^n$ such that
\begin{equation}\label{05122018E2}
	\exists c_0>0, \forall\xi\in\Omega_0,\quad\sum_{j=0}^r\vert Q(B^T)^j\xi\vert^2\geq c_0\vert\xi\vert^2.
\end{equation}
Let $\lambda_0>0$, $0\le q_{r-2}\le q_{r-1}\le1$ and $s_{r-1}\geq s_{r-2}\geq0$ be some non-negative real numbers. By setting $q_j = s_j = 0$ when $j\le r-3$, we define recursively the positive real numbers $\lambda_1,\ldots,\lambda_r>0$ by
\begin{equation}\label{22042020E1}
	\forall j\in\{0,\ldots,r-1\},\quad\lambda_{j+1}\bigg(\frac{1-q_j}{\lambda_j}+1\bigg) = 1+s_j.
\end{equation}
When the family $(\lambda_0,\ldots,\lambda_r)$ is decreasing and satisfies 
$$\forall j\in\{2,\ldots,r\},\quad \frac{\lambda_{j-1}}{\lambda_{j-2}}+\frac{\lambda_{j-1}}{\lambda_j}\le 2.$$
when $r\geq2$, there exist some positive constants $c_1,c_2,c_3>0$ and a real-valued symbol $g\in C^{\infty}(\mathbb R^n)$ satisfying the two estimates
\begin{equation}\label{05122018E3}
	\forall\xi\in\Omega_0,\quad \vert g(\xi)\vert\le c_1\sum_{j=0}^{r-1}\langle Q(B^T)^j\xi\rangle^{q_j}\langle\xi\rangle^{s_j},
\end{equation}
and
\begin{equation}\label{28032019E3}
	\forall\xi\in\Omega_0,\quad \langle\xi\rangle^{\lambda_r}\le c_2\langle Q\xi\rangle^{\lambda_0} + c_3B^T\xi\cdot\nabla_{\xi}g(\xi).
\end{equation}
\end{prop}

\begin{proof} Let $B$ and $Q$ be some real $n\times n$ matrices, with $Q$ symmetric positive semidefinite. We shall prove Proposition \ref{05122018P1} by induction on the positive integer $r\geq1$ such that \eqref{05122018E2} holds.
\newline

\noindent\textbf{A first multiplier.} Let $r\geq1$ be a positive integer. Let $\lambda_{r-1}>0$, $0\le q_{r-1}\le1$ and $s_{r-1}\geq0$ be some non-negative real numbers. We define the positive real number $\lambda_r>0$ by the relation
\begin{equation}\label{12122018E9}
	\lambda_r\bigg(\frac{1-q_{r-1}}{\lambda_{r-1}}+1\bigg) = 1+s_{r-1},
\end{equation}
and assume that $\lambda_r<\lambda_{r-1}$. First of all, we begin by studying the symbol $g_r\in C^{\infty}(\mathbb R^n)$ defined for all $\xi\in\mathbb R^n$ by
\begin{equation}\label{05122018E7}
	g_r(\xi) = \frac{Q(B^T)^{r-1}\xi\cdot Q(B^T)^r\xi}{\langle\xi\rangle^{2-\lambda_r}}\ \psi\bigg(\frac{\vert Q(B^T)^{r-1}\xi\vert^2}{\langle\xi\rangle^{2\lambda_r/\lambda_{r-1}}}\bigg),
\end{equation}
where $\psi\in C^{\infty}(\mathbb R,[0,1])$ is  smooth function satisfying 
\begin{equation}\label{23072019E5}
	\text{$\psi = 1$ on $\big\{x\in\mathbb R : \vert x\vert\le1/2\big\}$},\quad \Supp\psi\subset\big\{x\in\mathbb R : \vert x\vert\le1\big\}.
\end{equation}
It follows from \eqref{12122018E9} and Cauchy-Schwarz' inequality that the symbol $g_r$ satisfies the following estimate for all $\xi\in\mathbb R^n$,
\begin{align}\label{12122018E1}
	\langle Q(B^T)^{r-1}\xi\rangle^{-q_{r-1}}\vert g_r(\xi)\vert & \le\frac{\langle Q(B^T)^{r-1}\xi\rangle^{1-q_{r-1}}\langle Q(B^T)^r\xi\rangle}{\langle\xi\rangle^{2-\lambda_r}}
	\bigg\vert\psi\bigg(\frac{\vert Q(B^T)^{r-1}\xi\vert^2}{\langle\xi\rangle^{2\lambda_r/\lambda_{r-1}}}\bigg)\bigg\vert \\
	& \lesssim\frac{\langle\xi\rangle^{\lambda_r(1-q_{r-1})/\lambda_{r-1}+1}}{\langle\xi\rangle^{2-\lambda_r}} \simeq\langle\xi\rangle^{s_{r-1}}. \nonumber
\end{align}
We shall study the derivatives of the symbol $g_r$. By using that for all $\xi\in\mathbb R^n$,
$$B^T\xi\cdot\nabla_{\xi}(Q(B^T)^{r-1}\xi\cdot Q(B^T)^r\xi)
= \vert Q(B^T)^r\xi\vert^2 + Q(B^T)^{r-1}\xi\cdot Q(B^T)^{r+1}\xi,$$
we deduce from the definition \eqref{05122018E7} of $g_r$ and a direct computation that for all $\xi\in\mathbb R^n$,
\begin{equation}\label{05122018E10}
	B^T\xi\cdot\nabla_{\xi}g_r(\xi) = \frac{\vert Q(B^T)^r\xi\vert^2}{\langle\xi\rangle^{2-\lambda_r}}\psi\bigg(\frac{\vert Q(B^T)^{r-1}\xi\vert^2}{\langle\xi\rangle^{2\lambda_r/\lambda_{r-1}}}\bigg) + A_1(\xi) + A_2(\xi) + A_3(\xi),
\end{equation}
where the three terms $A_1(\xi)$, $A_2(\xi)$ and $A_3(\xi)$ are respectively given by
\begin{align}
	& A_1(\xi) = \frac{Q(B^T)^{r-1}\xi\cdot Q(B^T)^{r+1}\xi}{\langle\xi\rangle^{2-\lambda_r}}\ \psi\bigg(\frac{\vert Q(B^T)^{r-1}\xi\vert^2}{\langle\xi\rangle^{2\lambda_r/\lambda_{r-1}}}\bigg),\label{12122018E2} \\[5pt]
	& A_2(\xi) = (Q(B^T)^{r-1}\xi\cdot Q(B^T)^r\xi)\big(B^T\xi\cdot\nabla_{\xi}\langle\xi\rangle^{-2+\lambda_r}\big)\psi\bigg(\frac{\vert Q(B^T)^{r-1}\xi\vert^2}{\langle\xi\rangle^{2\lambda_r/\lambda_{r-1}}}\bigg), \label{12122018E3}
\end{align}
and
\begin{equation}
	A_3(\xi) = \frac{Q(B^T)^{r-1}\xi\cdot Q(B^T)^r\xi}{\langle\xi\rangle^{2-\lambda_r}}B^T\xi\cdot\nabla_{\xi}\bigg(\psi\bigg(\frac{\vert Q(B^T)^{r-1}\xi\vert^2}{\langle\xi\rangle^{2\lambda_r/\lambda_{r-1}}}\bigg)\bigg). \label{12122018E4}
\end{equation}
We will focus on each term. First, notice that the two first ones $A_1(\xi)$ and $A_2(\xi)$ satisfy the following estimate:
\begin{equation}\label{05122018E11}
	\forall\xi\in\mathbb R^n,\quad \vert A_1(\xi)\vert + \vert A_2(\xi)\vert\lesssim\langle Q(B^T)^{r-1}\xi\rangle^{q_{r-1}}\langle\xi\rangle^{s_{r-1}}.
\end{equation}
Indeed, on the one hand, we deduce from \eqref{12122018E9}, \eqref{12122018E2} and Cauchy-Schwarz' inequality that for all $\xi\in\mathbb R^n$,
\begin{align*}
	\langle Q(B^T)^{r-1}\xi\rangle^{-q_{r-1}}\vert A_1(\xi)\vert & \le\frac{\langle Q(B^T)^{r-1}\xi\rangle^{1-q_{r-1}}\langle Q(B^T)^{r+1}\xi\rangle}{\langle\xi\rangle^{2-\lambda_r}}
\bigg\vert\psi\bigg(\frac{\vert Q(B^T)^{r-1}\xi\vert^2}{\langle\xi\rangle^{2\lambda_r/\lambda_{r-1}}}\bigg)\bigg\vert \\
	& \lesssim\frac{\langle\xi\rangle^{\lambda_r(1-q_{r-1})/\lambda_{r-1}+1}}{\langle\xi\rangle^{2-\lambda_r}} \simeq \langle\xi\rangle^{s_{r-1}},
\end{align*}
since the function $\psi$ is supported in $[-1,1]$. On the other hand, we get from a direct computation that for all $\xi\in\mathbb R^n$,
$$\big\vert B^T\xi\cdot\nabla_{\xi}\langle\xi\rangle^{-2+\lambda_r}\big\vert
= \big\vert(-2+\lambda_r)\langle\xi\rangle^{-4+\lambda_r}(B^T\xi\cdot\xi)\big\vert
\lesssim\langle\xi\rangle^{-2+\lambda_r},$$
and the same ingredients imply that for all $\xi\in\mathbb R^n$,
$$\langle Q(B^T)^{r-1}\xi\rangle^{-q_{r-1}}\vert A_2(\xi)\vert\lesssim\frac{\langle Q(B^T)^{r-1}\xi\rangle^{1-q_{r-1}}\langle Q(B^T)^r\xi\rangle}{\langle\xi\rangle^{2-\lambda_r}}
\bigg\vert\psi\bigg(\frac{\vert Q(B^T)^{r-1}\xi\vert^2}{\langle\xi\rangle^{2\lambda_r/\lambda_{r-1}}}\bigg)\bigg\vert\lesssim\langle\xi\rangle^{s_{r-1}}.$$
This proves that \eqref{05122018E11} holds. Moreover, it follows from Cauchy-Schwarz' inequality and Lemma \ref{29032019L1} that for all $\xi\in\mathbb R^n$,
\begin{align}\label{12122018E7}
	\vert A_3(\xi)\vert & \lesssim\frac{\vert Q(B^T)^{r-1}\xi\vert\vert Q(B^T)^r\xi\vert}{\langle\xi\rangle^{2-\lambda_r}}\langle\xi\rangle^{1-\lambda_r/\lambda_{r-1}}
	\bigg\vert\psi'\bigg(\frac{\vert Q(B^T)^{r-1}\xi\vert^2}{\langle\xi\rangle^{2\lambda_r/\lambda_{r-1}}}\bigg)\bigg\vert \\
	& \lesssim \frac{\langle\xi\rangle^{\lambda_r/\lambda_{r-1}+1}\langle\xi\rangle^{1-\lambda_r/\lambda_{r-1}}}{\langle\xi\rangle^{2-\lambda_r}}
	\bigg\vert\psi'\bigg(\frac{\vert Q(B^T)^{r-1}\xi\vert^2}{\langle\xi\rangle^{2\lambda_r/\lambda_{r-1}}}\bigg)\bigg\vert
	= \langle\xi\rangle^{\lambda_r}\bigg\vert\psi'\bigg(\frac{\vert Q(B^T)^{r-1}\xi\vert^2}{\langle\xi\rangle^{2\lambda_r/\lambda_{r-1}}}\bigg)\bigg\vert \nonumber \\
	& \lesssim\vert Q(B^T)^{r-1}\xi\vert^{\lambda_{r-1}}\bigg\vert\psi'\bigg(\frac{\vert Q(B^T)^{r-1}\xi\vert^2}{\langle\xi\rangle^{2\lambda_r/\lambda_{r-1}}}\bigg)\bigg\vert. \nonumber
\end{align}
We therefore deduce from \eqref{05122018E10}, \eqref{05122018E11} and \eqref{12122018E7} that for all $\xi\in\mathbb R^n$,
\begin{multline}\label{05122018E13}
	\frac{\vert Q(B^T)^r\xi\vert^2}{\langle\xi\rangle^{2-\lambda_r}}\psi\bigg(\frac{\vert Q(B^T)^{r-1}\xi\vert^2}{\langle\xi\rangle^{2\lambda_r/\lambda_{r-1}}}\bigg)\lesssim
	\vert Q(B^T)^{r-1}\xi\vert^{\lambda_{r-1}}\bigg\vert\psi'\bigg(\frac{\vert Q(B^T)^{r-1}\xi\vert^2}{\langle\xi\rangle^{2\lambda_r/\lambda_{r-1}}}\bigg)\bigg\vert \\
	+ \langle Q(B^T)^{r-1}\xi\rangle^{q_{r-1}}\langle\xi\rangle^{s_{r-1}} + B^T\xi\cdot\nabla_{\xi}g_r(\xi).
\end{multline}

\medskip

\noindent\textbf{Basic case.} We assume that there exists an open subset $\Omega_0\subset\mathbb R^n$ such that the estimate \eqref{05122018E2} holds with $r=1$. Keeping the notations introduced in the previous step, we also assume that $\lambda_0$ and $\lambda_1$ satisfy $\lambda_0<\lambda_1$. We aim first at proving that for all $\xi\in\Omega_0$,
\begin{equation}\label{03042019E5}
	\langle\xi\rangle^{\lambda_1}\lesssim\langle Q\xi\rangle^{\lambda_0} + \langle Q\xi\rangle^{q_0}\langle\xi\rangle^{s_0} + B^T\xi\cdot\nabla_{\xi}g_1(\xi).
\end{equation}
When $r=1$, the inequality \eqref{05122018E13} writes for all $\xi\in\mathbb R^n$ as
$$\frac{\vert QB^T\xi\vert^2}{\langle\xi\rangle^{2-\lambda_1}}\psi\bigg(\frac{\vert Q\xi\vert^2}{\langle\xi\rangle^{2\lambda_1/\lambda_0}}\bigg)
\lesssim\vert Q\xi\vert^{\lambda_0}\bigg\vert\psi'\bigg(\frac{\vert Q\xi\vert^2}{\langle\xi\rangle^{2\lambda_1/\lambda_0}}\bigg)\bigg\vert + \langle Q\xi\rangle^{q_0}\langle\xi\rangle^{s_0} + B^T\xi\cdot\nabla_{\xi}g_1(\xi).$$
Since the function $\psi'$ is bounded, we deduce that for all $\xi\in\mathbb R^n$,
\begin{equation}\label{13122018E1}
	\frac{\vert QB^T\xi\vert^2}{\langle\xi\rangle^{2-\lambda_1}}\psi\bigg(\frac{\vert Q\xi\vert^2}{\langle\xi\rangle^{2\lambda_1/\lambda_0}}\bigg)
	\lesssim\vert Q\xi\vert^{\lambda_0} + \langle Q\xi\rangle^{q_0}\langle\xi\rangle^{s_0} + B^T\xi\cdot\nabla_{\xi}g_1(\xi).
\end{equation}
We consider a function $w\in C^{\infty}(\mathbb R,[0,1])$ satisfying
\begin{equation}\label{23072019E4}
	\text{$w = 1$ on $\big\{x\in\mathbb R : \vert x\vert\geq1/2\big\}$},\quad \Supp w\subset\big\{x\in\mathbb R : \vert x\vert\geq1/4\big\}.
\end{equation}
By definition of the function $w$, we notice that for all $\xi\in\mathbb R^n$,
\begin{equation}\label{13122018E2}
	\langle\xi\rangle^{\lambda_1}w\bigg(\frac{\vert Q\xi\vert^2}{\langle\xi\rangle^{2\lambda_1/\lambda_0}}\bigg)\lesssim\vert Q\xi\vert^{\lambda_0}.
\end{equation}
By summing the estimates \eqref{13122018E1} and \eqref{13122018E2}, we deduce that for all $\xi\in\mathbb R^n$,
\begin{equation}\label{05122018E16}
	\frac{\vert QB^T\xi\vert^2}{\langle\xi\rangle^{2-\lambda_1}}\psi\bigg(\frac{\vert Q\xi\vert^2}{\langle\xi\rangle^{2\lambda_1/\lambda_0}}\bigg)
	+\langle\xi\rangle^{\lambda_1}w\bigg(\frac{\vert Q\xi\vert^2}{\langle\xi\rangle^{2\lambda_1/\lambda_0}}\bigg)
	\lesssim\langle Q\xi\rangle^{\lambda_0} + \langle Q\xi\rangle^{q_0}\langle\xi\rangle^{s_0} + B^T\xi\cdot\nabla_{\xi}g_1(\xi).
\end{equation}
Let $\xi\in\Omega_0$ satisfying $\vert\xi\vert\geq r_0$, where $r_0\geq1$ is a positive constant whose value will adjusted later. We will distinguish two regions in $\Omega_0$ and we first assume that
\begin{equation}\label{13122018E3}
	\frac{\vert Q\xi\vert^2}{\langle\xi\rangle^{2\lambda_1/\lambda_0}}\le\frac12.
\end{equation}
By using \eqref{05122018E2}, \eqref{13122018E3} and the fact that $\vert\xi\vert\geq r_0$ with $r_0\geq1$, we obtain that
$$\vert QB^T\xi\vert^2\geq c_0\vert\xi\vert^2-\vert Q\xi\vert^2\geq\frac{c_0r^2_0}{1+r^2_0}\langle\xi\rangle^2-\frac12\langle\xi\rangle^{2\lambda_1/\lambda_0}
\geq\frac{c_0}2\langle\xi\rangle^2-\frac12\langle\xi\rangle^{2\lambda_1/\lambda_0}.$$
Let us recall that $\lambda_1>\lambda_0$ by assumption. We can therefore choose $r_0\gg1$ large enough so that $\vert QB^T\xi\vert^2\gtrsim\langle\xi\rangle^2$. As a consequence, the estimate \eqref{05122018E16} becomes
$$\langle\xi\rangle^{\lambda_1}\lesssim\frac{\vert QB^T\xi\vert^2}{\langle\xi\rangle^{2-\lambda_1}}
\lesssim\langle Q\xi\rangle^{\lambda_0} + \langle Q\xi\rangle^{q_0}\langle\xi\rangle^{s_0} + B^T\xi\cdot\nabla_{\xi}g_1(\xi),$$
since
$$\psi\bigg(\frac{\vert Q\xi\vert^2}{\langle\xi\rangle^{2\lambda_1/\lambda_0}}\bigg)=1\quad\text{and}\quad\langle\xi\rangle^{\lambda_1}w\bigg(\frac{\vert Q\xi\vert^2}{\langle\xi\rangle^{2\lambda_1/\lambda_0}}\bigg)\geq0,$$
from the definition \eqref{23072019E5} of the function $\psi$ and \eqref{13122018E3}. Then, we assume that $\xi\in\Omega_0$ satisfies
\begin{equation}\label{13122018E4}
	\frac{\vert Q\xi\vert^2}{\langle\xi\rangle^{2\lambda_1/\lambda_0}}\geq\frac12.
\end{equation}
In this case,
$$\frac{\vert QB^T\xi\vert^2}{\langle\xi\rangle^{2-\lambda_1}}\psi\bigg(\frac{\vert Q\xi\vert^2}{\langle\xi\rangle^{2\lambda_1/\lambda_0}}\bigg)\geq0\quad \text{and}\quad w\bigg(\frac{\vert Q\xi\vert^2}{\langle\xi\rangle^{2\lambda_1/\lambda_0}}\bigg)=1,$$
from the definition \eqref{23072019E4} of the function $w$, and we obtain from \eqref{05122018E16} anew that
$$\langle\xi\rangle^{\lambda_1}\lesssim\langle Q\xi\rangle^{\lambda_0} + \langle Q\xi\rangle^{q_0}\langle\xi\rangle^{s_0} + B^T\xi\cdot\nabla_{\xi}g_1(\xi).$$
We proved that \eqref{03042019E5} holds when $\vert\xi\vert\geq r_0$. Since all the functions involved in this estimate are continuous, it is also valid when $\vert\xi\vert\le r_0$ from a compactness argument, up to increasing all the constants. Moreover, the assumption $\lambda_1>\lambda_0$ and a straightforward calculus show that $\lambda_1>q_0+s_0$, and Young's inequality implies that for all $\varepsilon>0$, there exists a positive constant $c_{\varepsilon}>0$ such that for all $\xi\in\mathbb R^n$,
$$\langle Q\xi\rangle^{q_0}\langle\xi\rangle^{s_0}\lesssim\langle\xi\rangle^{q_0+s_0}\lesssim\varepsilon\langle\xi\rangle^{\lambda_1} + c_{\varepsilon}.$$
Adjusting the value of $\varepsilon>0$, we deduce that for all $\xi\in\Omega_0$,
$$\langle\xi\rangle^{\lambda_1}\lesssim\langle Q\xi\rangle^{\lambda_0} + B^T\xi\cdot\nabla_{\xi}g_1(\xi).$$
This ends the proof of Proposition \ref{05122018P1} in the case when $r=1$, and it proves the induction hypothesis in the basic case.
\newline

\noindent\textbf{Induction.} We now consider $r\geq2$ a positive integer and assume that Proposition \ref{05122018P1} is proven for $r-1$. We also assume that there exists an open subset $\Omega_0\subset\mathbb R^n$ such that the estimate \eqref{05122018E2} holds. Let $\lambda_0>0$, $0\le q_{r-2}\le q_{r-1}\le1$ and $s_{r-1}\geq s_{r-2}\geq0$ be some non-negative real numbers. By setting $q_j = s_j = 0$ when $j\le r-3$, we define recursively the positive real numbers $\lambda_1,\ldots,\lambda_r>0$ by \eqref{22042020E1}. Let us assume that the family $(\lambda_0,\ldots,\lambda_r)$ is decreasing and satisfies \eqref{15042020E1}. The purpose is to obtain the existence of a smooth function $w\in C^{\infty}(\mathbb R,[0,1])$ with similar properties as the one defined in \eqref{23072019E4}, with possibly different numerical values for its support localisation, and a smooth real-valued symbol $G_r\in C^{\infty}(\mathbb R^n)$ satisfying 
\begin{equation}\label{26062020E1}
	\forall\xi\in\Omega_0,\quad \vert G_r(\xi)\vert\lesssim\langle Q(B^T)^{r-2}\xi\rangle^{q_{r-2}}\langle\xi\rangle^{s_{r-2}} + \langle Q(B^T)^{r-1}\xi\rangle^{q_{r-1}}\langle\xi\rangle^{s_{r-1}},
\end{equation}
and such that for all $\xi\in\Omega_0$,
\begin{multline}\label{23072019E6}
	\langle\xi\rangle^{\lambda_r}
	\lesssim\vert Q(B^T)^{r-1}\xi\vert^{\lambda_{r-1}}w\bigg(\frac{\vert Q(B^T)^{r-1}\xi\vert^2}{\langle\xi\rangle^{2\lambda_r/\lambda_{r-1}}}\bigg) + \langle Q\xi\rangle^{\lambda_0} \\[5pt]
	+ \langle Q(B^T)^{r-2}\xi\rangle^{q_{r-2}}\langle\xi\rangle^{s_{r-2}} + \langle Q(B^T)^{r-1}\xi\rangle^{q_{r-1}}\langle\xi\rangle^{s_{r-1}} + B^T\xi\cdot\nabla_{\xi}G_r(\xi).
\end{multline}
This estimate combined with \eqref{23042020E2} and \eqref{23042020E1} (proven in the following step, a little further) ends the proof of Theorem \ref{05122018P1}. Indeed, given $\varepsilon>0$, it follows from these two estimates that there exists a smooth real-valued symbol $p_{r,\varepsilon}\in C^{\infty}(\mathbb R^n)$ such that for all $\xi\in\mathbb R^n$,
$$\vert p_{r,\varepsilon}(\xi)\vert\lesssim\langle Q(B^T)^{r-2}\xi\rangle^{q_{r-2}}\langle\xi\rangle^{s_{r-2}},$$
and
$$\vert Q(B^T)^{r-1}\xi\vert^{\lambda_{r-1}}w\bigg(\frac{\vert Q(B^T)^{r-1}\xi\vert^2}{\langle\xi\rangle^{2\lambda_r/\lambda_{r-1}}}\bigg)
\lesssim\langle Q\xi\rangle^{\lambda_0} + \varepsilon\langle\xi\rangle^{\lambda_r} + B^T\xi\cdot\nabla_{\xi}p_{r,\varepsilon}(\xi).$$
Adjusting the value of $\varepsilon>0$ and setting $g = G_r + p_{r,\varepsilon}$, which is a regular real-valued symbol satisfying \eqref{05122018E3}, we obtain that for all $\xi\in\Omega_0$,  
$$\langle\xi\rangle^{\lambda_r}\lesssim\langle Q\xi\rangle^{\lambda_0} + \langle Q(B^T)^{r-2}\xi\rangle^{q_{r-2}}\langle\xi\rangle^{s_{r-2}} + \langle Q(B^T)^{r-1}\xi\rangle^{q_{r-1}}\langle\xi\rangle^{s_{r-1}} + B^T\xi\cdot\nabla_{\xi}g(\xi).$$
Since we assumed that $\lambda_{r-1}>\lambda_r$, the estimates $\lambda_r>q_{r-1}+s_{r-1}\geq q_{r-2}+s_{r-2}$ are valid from Lemma \ref{10042019L1}, and Young's inequality implies that for all $\eta>0$, there exists a positive constant $c_{\eta}>0$ such that for all $\xi\in\mathbb R^n$,
\begin{multline*}
	\langle Q(B^T)^{r-2}\xi\rangle^{q_{r-2}}\langle\xi\rangle^{s_{r-2}} + \langle Q(B^T)^{r-1}\xi\rangle^{q_{r-1}}\langle\xi\rangle^{s_{r-1}} \\[5pt]
	\lesssim\langle\xi\rangle^{q_{r-2}+s_{r-2}} + \langle\xi\rangle^{q_{r-1}+s_{r-1}} 
	\lesssim\eta\langle\xi\rangle^{\lambda_r} + c_{\eta},
\end{multline*}
Adjusting the value of $\eta>0$, Plancherel's theorem then implies that for all $\xi\in\Omega_0$,
$$\langle\xi\rangle^{\lambda_r}\lesssim \langle Q\xi\rangle^{\lambda_0} + B^T\xi\cdot\nabla_{\xi}g(\xi),$$
which is the required estimate.

To establish \eqref{23072019E6}, we distinguish two cases. On the one hand, in the situation where
\begin{equation}\label{09042019E10}
	\forall c_1\in(0,c_0),\forall\xi\in\Omega_0,\quad \vert Q(B^T)^r\xi\vert^2\geq c_1\vert\xi\vert^2,
\end{equation}
where $c_0>0$ denotes the positive constant appearing in \eqref{05122018E2}, we get from \eqref{05122018E13} that for all $\xi\in\Omega_0$ satisfying $\vert\xi\vert\geq1$,
\begin{multline}\label{23072019E8}
	\langle\xi\rangle^{\lambda_r}\psi\bigg(\frac{\vert Q(B^T)^{r-1}\xi\vert^2}{\langle\xi\rangle^{2\lambda_r/\lambda_{r-1}}}\bigg)
	\lesssim\frac{\vert Q(B^T)^r\xi\vert^2}{\langle\xi\rangle^{2-\lambda_r}}\psi\bigg(\frac{\vert Q(B^T)^{r-1}\xi\vert^2}{\langle\xi\rangle^{2\lambda_r/\lambda_{r-1}}}\bigg) \\[5pt]
	\lesssim\vert Q(B^T)^{r-1}\xi\vert^{\lambda_{r-1}}\bigg\vert \psi'\bigg(\frac{\vert Q(B^T)^{r-1}\xi\vert^2}{\langle\xi\rangle^{2\lambda_r/\lambda_{r-1}}}\bigg)\bigg\vert
	+ \langle Q(B^T)^{r-1}\xi\rangle^{q_{r-1}}\langle\xi\rangle^{s_{r-1}} + B^T\xi\cdot\nabla_{\xi}g_r(\xi).
\end{multline}
Then, observing from \eqref{23072019E5} that the function $1-\psi$ is supported in $(-\infty,-1/2]\cup[1/2,+\infty)$, we obtain the following estimate for all $\xi\in\Omega_0$ satisfying $\vert\xi\vert\geq1$, 
\begin{align}\label{23072019E7}
	\langle\xi\rangle^{\lambda_r} & = \langle\xi\rangle^{\lambda_r}\psi\bigg(\frac{\vert Q(B^T)^{r-1}\xi\vert^2}{\langle\xi\rangle^{2\lambda_r/\lambda_{r-1}}}\bigg)
	+\langle\xi\rangle^{\lambda_r}(1-\psi)\bigg(\frac{\vert Q(B^T)^{r-1}\xi\vert^2}{\langle\xi\rangle^{2\lambda_r/\lambda_{r-1}}}\bigg) \\[5pt]
	& \lesssim \langle\xi\rangle^{\lambda_r}\psi\bigg(\frac{\vert Q(B^T)^{r-1}\xi\vert^2}{\langle\xi\rangle^{2\lambda_r/\lambda_{r-1}}}\bigg)
	+\vert Q(B^T)^{r-1}\xi\vert^{\lambda_{r-1}}(1-\psi)\bigg(\frac{\vert Q(B^T)^{r-1}\xi\vert^2}{\langle\xi\rangle^{2\lambda_r/\lambda_{r-1}}}\bigg). \nonumber
\end{align}
Considering a function $w\in C^{\infty}(\mathbb R,[0,1])$ with similar properties as the one defined in \eqref{23072019E4}, with possibly different numerical values for its support localisation, such that $\vert\psi'\vert,1-\psi\le w$, which is possible in view of \eqref{23072019E5}, we deduce from \eqref{23072019E8} and \eqref{23072019E7} that for all $\xi\in\Omega_0$ satisfying $\vert\xi\vert\geq1$,
\begin{equation}\label{23072019E23}
	\langle\xi\rangle^{\lambda_r}\lesssim\vert Q(B^T)^{r-1}\xi\vert^{\lambda_{r-1}}w\bigg(\frac{\vert Q(B^T)^{r-1}\xi\vert^2}{\langle\xi\rangle^{2\lambda_r/\lambda_{r-1}}}\bigg) + \langle Q(B^T)^{r-1}\xi\rangle^{q_{r-1}}\langle\xi\rangle^{s_{r-1}} + B^T\xi\cdot\nabla_{\xi}g_r(\xi).
\end{equation}
This estimate can be extended to all $\xi\in\Omega_0$ satisfying $\vert\xi\vert\le1$, all the functions involved being continuous. Since the estimate \eqref{12122018E1} holds, this proves \eqref{26062020E1} and \eqref{23072019E6}.

On the other hand, when the estimate \eqref{09042019E10} does not hold, we may find some positive constants $c_1, c_2>0$, such that the estimate
\begin{equation}\label{05122018E20}
	\sum_{j=0}^{r-1}\vert Q(B^T)^j\xi\vert^2\geq c_2\vert\xi\vert^2,
\end{equation}
holds on the non-empty open set
\begin{equation}\label{06122018E1}
	\Omega_1 = \big\{\xi\in\mathbb R^n : \vert Q(B^T)^r\xi\vert^2<c_1\vert\xi\vert^2\big\}\cap\Omega_0.
\end{equation}
We deduce from the induction hypothesis that there exists a smooth real-valued symbol $h_{r-1}\in C^{\infty}(\mathbb R^n)$ such that
\begin{equation}\label{06122018E2}
	\forall\xi\in\Omega_1,\quad \vert h_{r-1}(\xi)\vert\lesssim\langle Q(B^T)^{r-2}\xi\rangle^{q_{r-2}}\langle\xi\rangle^{s_{r-2}},
\end{equation}
since $q_j = s_j = 0$ when $j\le r-3$, and satisfying that for all $\xi\in\Omega_1$, 
\begin{equation}\label{06122018E3}
	\langle\xi\rangle^{\lambda_{r-1}}\lesssim\langle Q\xi\rangle^{\lambda_0}+B^T\xi\cdot\nabla_{\xi}h_{r-1}(\xi).
\end{equation}
We choose $\psi_0$ and $w_0$ some $C^{\infty}(\mathbb R,[0,1])$ functions satisfying similar properties as the functions defined in \eqref{23072019E5} and \eqref{23072019E4} respectively, with possibly different positive numerical values for their support localizations, such that
\begin{equation}\label{06122018E4}
	\Supp\bigg(\psi_0\bigg(\frac{\vert Q(B^T)^r\cdot\vert^2}{\vert\cdot\vert^2}\bigg)w_0(\vert\cdot\vert^2)\bigg)\subset\big\{\xi\in\mathbb R^n : \vert Q(B^T)^r\xi\vert^2<c_1\vert\xi\vert^2\big\}.
\end{equation}
It follows from \eqref{06122018E3} that for all $\xi\in\Omega_0$,
$$\psi_0\bigg(\frac{\vert Q(B^T)^r\xi\vert^2}{\vert\xi\vert^2}\bigg)w_0(\vert\xi\vert^2)\langle\xi\rangle^{\lambda_{r-1}}
\lesssim\psi_0\bigg(\frac{\vert Q(B^T)^r\xi\vert^2}{\vert\xi\vert^2}\bigg)w_0(\vert\xi\vert^2)\big(\langle Q\xi\rangle^{\lambda_0} + B^T\xi\cdot\nabla_{\xi}h_{r-1}(\xi)\big).$$
Since the functions $\psi_0$ and $w_0$ are bounded, this inequality shows that for all $\xi\in\Omega_0$,
\begin{equation}\label{13122018E10}
	\psi_0\bigg(\frac{\vert Q(B^T)^r\xi\vert^2}{\vert\xi\vert^2}\bigg)w_0(\vert\xi\vert^2)\langle\xi\rangle^{\lambda_{r-1}} \lesssim
	\langle Q\xi\rangle^{\lambda_0} + \psi_0\bigg(\frac{\vert Q(B^T)^r\xi\vert^2}{\vert\xi\vert^2}\bigg)w_0(\vert\xi\vert^2)B^T\xi\cdot\nabla_{\xi}h_{r-1}(\xi).
\end{equation}
We consider the symbol $\tilde g_{r-1}$ defined for all $\xi\in\Omega_0$ by
\begin{equation}\label{13122018E7}
	\tilde g_{r-1}(\xi) = \psi_0\bigg(\frac{\vert Q(B^T)^r\xi\vert^2}{\vert\xi\vert^2}\bigg)w_0(\vert\xi\vert^2)h_{r-1}(\xi).
\end{equation}
Notice from \eqref{06122018E2} and by choice of the functions $\psi_0$ and $w_0$ that the symbol $\tilde g_{r-1}$ satisfies
\begin{equation}\label{23072019E21}
	\forall\xi\in\Omega_0,\quad\vert\tilde g_{r-1}(\xi)\vert\lesssim\langle Q(B^T)^{r-2}\xi\rangle^{q_{r-2}}\langle\xi\rangle^{s_{r-2}}.
\end{equation}
In order to reformulate the estimate \eqref{13122018E10} in term of the symbol $\tilde g_{r-1}$ instead of $h_{r-1}$, we compute its derivatives. We first notice from Leibniz' formula that for all $\xi\in\Omega_0$,
\begin{multline*}
	B^T\xi\cdot\nabla_{\xi}\tilde g_{r-1}(\xi) = B^T\xi\cdot\nabla_{\xi}\bigg(\psi_0\bigg(\frac{\vert Q(B^T)^r\xi\vert^2}{\vert\xi\vert^2}\bigg)w_0(\vert\xi\vert^2)\bigg)h_{r-1}(\xi) \\[5pt]
	+\psi_0\bigg(\frac{\vert Q(B^T)^r\xi\vert^2}{\vert\xi\vert^2}\bigg)w_0(\vert\xi\vert^2)B^T\xi\cdot\nabla_{\xi}h_{r-1}(\xi).
\end{multline*}
Let us check that
\begin{equation}\label{23072019E10}
	B^T\xi\cdot\nabla_{\xi}\bigg(\psi_0\bigg(\frac{\vert Q(B^T)^r\xi\vert^2}{\vert\xi\vert^2}\bigg)w_0(\vert\xi\vert^2)\bigg)\in L^{\infty}(\mathbb R^n).
\end{equation}
We directly compute that for all $\xi\in\mathbb R^n$,
\begin{align*}
	&\ B^T\xi\cdot\nabla_{\xi}\bigg(\psi_0\bigg(\frac{\vert Q(B^T)^r\xi\vert^2}{\vert\xi\vert^2}\bigg)\bigg)w_0(\vert\xi\vert^2) \\[5pt]
	= &\ \bigg(\frac{B^T\xi\cdot\nabla_{\xi}\vert Q(B^T)^r\xi\vert^2}{\vert\xi\vert^2}+\vert Q(B^T)^r\xi\vert^2B^T\xi\cdot\nabla_{\xi}\vert\xi\vert^{-2}\bigg)\psi'_0\bigg(\frac{\vert Q(B^T)^r\xi\vert^2}{\vert\xi\vert^2}\bigg)w_0(\vert\xi\vert^2) \\[5pt]
	= &\ \bigg(\frac{2Q(B^T)^{r+1}\xi\cdot Q(B^T)^r\xi}{\vert\xi\vert^2}-\frac{2\vert Q(B^T)^r\xi\vert^2B^T\xi\cdot\xi}{\vert\xi\vert^4}\bigg)\psi'_0\bigg(\frac{\vert Q(B^T)^r\xi\vert^2}{\vert\xi\vert^2}\bigg)w_0(\vert\xi\vert^2)\in L^{\infty}(\mathbb R^n).
\end{align*}
On the other hand, since the function $w'_0$ is compactly supported from \eqref{23072019E4}, we notice that for all $\xi\in\mathbb R^n$,
$$\psi_0\bigg(\frac{\vert Q(B^T)^r\xi\vert^2}{\vert\xi\vert^2}\bigg)B^T\xi\cdot\nabla_{\xi}w_0(\vert\xi\vert^2)
= \psi_0\bigg(\frac{\vert Q(B^T)^r\xi\vert^2}{\vert\xi\vert^2}\bigg)2(B^T\xi\cdot\xi) w'_0(\vert\xi\vert^2)\in L^{\infty}(\mathbb R^n).$$
These two estimates and Leibniz' formula imply that \eqref{23072019E10} holds. Since the symbol $h_{r-1}$ satisfies the estimate \eqref{06122018E2}, it follows from \eqref{13122018E10} and the definition \eqref{13122018E7} of the symbol $\tilde g_{r-1}$ that for all $\xi\in\Omega_0$,
\begin{equation}\label{09042019E16}
	\langle\xi\rangle^{\lambda_{r-1}}\psi_0\bigg(\frac{\vert Q(B^T)^r\xi\vert^2}{\vert\xi\vert^2}\bigg)w_0(\vert\xi\vert^2) \\ 
	\lesssim\langle Q\xi\rangle^{\lambda_0} + \langle Q(B^T)^{r-2}\xi\rangle^{q_{r-2}}\langle\xi\rangle^{s_{r-2}} + B^T\xi\cdot\nabla_{\xi}\tilde g_{r-1}(\xi).
\end{equation}
Let $G_r$ be the smooth real-valued symbol defined by $G_r = g_r + \tilde g_{r-1}$. Notice from \eqref{12122018E1} and \eqref{23072019E21} that for all $\xi\in\Omega_0$,
\begin{equation}
	\vert G_r(\xi)\vert\lesssim\langle Q(B^T)^{r-2}\xi\rangle^{q_{r-2}}\langle\xi\rangle^{s_{r-2}} + \langle Q(B^T)^{r-1}\xi\rangle^{q_{r-1}}\langle\xi\rangle^{s_{r-1}}.
\end{equation}
By summing the estimates \eqref{05122018E13} and \eqref{09042019E16}, we deduce that for all $\xi\in\Omega_0$,
\begin{multline}\label{06122018E7}
	\frac{\vert Q(B^T)^r\xi\vert^2}{\langle\xi\rangle^{2-\lambda_r}}\psi\bigg(\frac{\vert Q(B^T)^{r-1}\xi\vert^2}{\langle\xi\rangle^{2\lambda_r/\lambda_{r-1}}}\bigg) 
	+ \langle\xi\rangle^{\lambda_{r-1}}\psi_0\bigg(\frac{\vert Q(B^T)^r\xi\vert^2}{\vert\xi\vert^2}\bigg)w_0(\vert\xi\vert^2) \\[5pt]
	\lesssim \vert Q(B^T)^{r-1}\xi\vert^{\lambda_{r-1}}\bigg\vert\psi'\bigg(\frac{\vert Q(B^T)^{r-1}\xi\vert^2}{\langle\xi\rangle^{2\lambda_r/\lambda_{r-1}}}\bigg)\bigg\vert + \langle Q\xi\rangle^{\lambda_0} + \langle Q(B^T)^{r-2}\xi\rangle^{q_{r-2}}\langle\xi\rangle^{s_{r-2}} \\[5pt]
	+ \langle Q(B^T)^{r-1}\xi\rangle^{q_{r-1}}\langle\xi\rangle^{s_{r-1}} + B^T\xi\cdot\nabla_{\xi}G_r(\xi).
\end{multline}
We can now tackle the proof of the estimate \eqref{23072019E6} itself. We need to be more precise concerning the support localization of the functions $\psi_0$ and $w_0$. Let $\sigma_0,\sigma_1>0$ be two positive constants, with $s_0<c_1$, satisfying that
\begin{equation}\label{23072019E22}
	\text{$\psi_0=1$ on $\big\{x\in\mathbb R : \vert x\vert\le\sigma_0\big\}$}\quad\text{and}\quad\text{$w_0=1$ on $\big\{x\in\mathbb R : \vert x\vert\geq\sigma_1\big\}$.}
\end{equation}
Let $\xi\in\Omega_0$ satisfying $\vert\xi\vert\geq\sigma_1$. As in the basic case, we need to distinguish two regions in $\Omega_0$. On the one hand, when $\vert Q(B^T)^r\xi\vert^2\le \sigma_0\vert\xi\vert^2$, we deduce the following estimate from \eqref{23072019E22},
$$\langle\xi\rangle^{\lambda_{r-1}}\psi_0\bigg(\frac{\vert Q(B^T)^r\xi\vert^2}{\vert\xi\vert^2}\bigg)w_0(\vert\xi\vert^2) = \langle\xi\rangle^{\lambda_{r-1}}\geq\langle\xi\rangle^{\lambda_r},$$
since $\lambda_{r-1}>\lambda_r$ by assumption. The estimate \eqref{06122018E7} therefore writes as
\begin{multline*}
	\langle\xi\rangle^{\lambda_r}\lesssim\vert Q(B^T)^{r-1}\xi\vert^{\lambda_{r-1}}\bigg\vert\psi'\bigg(\frac{\vert Q(B^T)^{r-1}\xi\vert^2}{\langle\xi\rangle^{2\lambda_r/\lambda_{r-1}}}\bigg)\bigg\vert \\[5pt]
	+ \langle Q\xi\rangle^{\lambda_0} + \langle Q(B^T)^{r-2}\xi\rangle^{q_{r-2}}\langle\xi\rangle^{s_{r-2}} + \langle Q(B^T)^{r-1}\xi\rangle^{q_{r-1}}\langle\xi\rangle^{s_{r-1}} + B^T\xi\cdot\nabla_{\xi}G_r(\xi),
\end{multline*}
since
$$\frac{\vert Q(B^T)^r\xi\vert^2}{\langle\xi\rangle^{2-\lambda_r}}\psi\bigg(\frac{\vert Q(B^T)^{r-1}\xi\vert^2}{\langle\xi\rangle^{2\lambda_r/\lambda_{r-1}}}\bigg)\geq0.$$
On the other hand, when $\vert Q(B^T)^r\xi\vert^2\geq\sigma_0\vert\xi\vert^2$, we deduce that $\vert Q(B^T)^r\xi\vert^2\gtrsim\langle\xi\rangle^2$ according to $\vert\xi\vert\geq\sigma_1$ with $\sigma_1>0$, which proves that
$$\langle\xi\rangle^{\lambda_r}\lesssim\frac{\vert Q(B^T)^r\xi\vert^2}{\langle\xi\rangle^{2-\lambda_r}}.$$
The inequality \eqref{06122018E7} therefore writes in this case in the following way
\begin{multline*}
	\langle\xi\rangle^{\lambda_r}\psi\bigg(\frac{\vert Q(B^T)^{r-1}\xi\vert^2}{\langle\xi\rangle^{2\lambda_r/\lambda_{r-1}}}\bigg)
	\lesssim\vert Q(B^T)^{r-1}\xi\vert^{\lambda_{r-1}}\bigg\vert\psi'\bigg(\frac{\vert Q(B^T)^{r-1}\xi\vert^2}{\langle\xi\rangle^{2\lambda_r/\lambda_{r-1}}}\bigg)\bigg\vert \\[5pt]
	+ \langle Q\xi\rangle^{\lambda_0} + \langle Q(B^T)^{r-2}\xi\rangle^{q_{r-2}}\langle\xi\rangle^{s_{r-2}} + \langle Q(B^T)^{r-1}\xi\rangle^{q_{r-1}}\langle\xi\rangle^{s_{r-1}} + B^T\xi\cdot\nabla_{\xi}G_r(\xi),
\end{multline*}
since
$$\langle\xi\rangle^{\lambda_{r-1}}\psi_0\bigg(\frac{\vert Q(B^T)^r\xi\vert^2}{\vert\xi\vert^2}\bigg)w_0(\vert\xi\vert^2)\geq0.$$
Proceeding similarly as we have already done to establish the estimate \eqref{23072019E23}, that is, mimicking \eqref{23072019E7}, we deduce that
\begin{multline*}
	\langle\xi\rangle^{\lambda_r}\lesssim\vert Q(B^T)^{r-1}\xi\vert^{\lambda_{r-1}}w\bigg(\frac{\vert Q(B^T)^{r-1}\xi\vert^2}{\langle\xi\rangle^{2\lambda_r/\lambda_{r-1}}}\bigg) \\[5pt]
	+ \langle Q\xi\rangle^{\lambda_0} + \langle Q(B^T)^{r-2}\xi\rangle^{q_{r-2}}\langle\xi\rangle^{s_{r-2}} + \langle Q(B^T)^{r-1}\xi\rangle^{q_{r-1}}\langle\xi\rangle^{s_{r-1}} + B^T\xi\cdot\nabla_{\xi}G_r(\xi),
\end{multline*}
where $w\in C^{\infty}(\mathbb R,[0,1])$ is a smooth function with the following properties
\begin{equation}\label{23072019E11}
	\text{$w = 1$ on $\big\{x\in\mathbb R : \vert x\vert\geq c_1\big\}$}\quad\text{and}\quad \Supp w\subset\big\{x\in\mathbb R : \vert x\vert\geq c_2\big\},
\end{equation}
where $c_1>c_2>0$, such that $\vert\psi'\vert\le w$ and $1-\psi\le w$. This disjunction of cases proves that the estimate \eqref{23072019E6} holds when $\vert\xi\vert\geq\sigma_1$. Moreover, since the all functions implied are continuous, \eqref{23072019E6} remains for all $\xi\in\Omega_0$ satisfying $\vert\xi\vert<\sigma_1$, up to grow all the positive constants at play.
\newline

\noindent\textbf{An instrumental estimate.} In order to end the proof of Proposition \ref{05122018P1}, it remains to check that for all $\varepsilon>0$, there exists a smooth real-valued symbol $p_{r,\varepsilon}\in C^{\infty}(\mathbb R^n)$ satisfying that for all $\xi\in\mathbb R^n$,
\begin{equation}\label{23042020E2}
	\vert p_{r,\varepsilon}(\xi)\vert\lesssim\langle Q(B^T)^{r-2}\xi\rangle^{q_{r-2}}\langle\xi\rangle^{s_{r-2}},
\end{equation}
and
\begin{equation}\label{23042020E1}
	\vert Q(B^T)^{r-1}\xi\vert^{\lambda_{r-1}}w\bigg(\frac{\vert Q(B^T)^{r-1}\xi\vert^2}{\langle\xi\rangle^{2\lambda_r/\lambda_{r-1}}}\bigg)
	\lesssim\langle Q\xi\rangle^{\lambda_0} + \varepsilon\langle\xi\rangle^{\lambda_r} + B^T\xi\cdot\nabla_{\xi}p_{r,\varepsilon}(\xi).
\end{equation}
where $w\in C^{\infty}(\mathbb R,[0,1])$ is the function defined in \eqref{23072019E11}. We fix $\varepsilon>0$ a small positive real number for the rest of this proof.
\newline

\noindent\textit{Introduction.} Let $\Gamma_2,\ldots,\Gamma_r\geq1$ be some large positive constants whose values will be chosen later on. Setting $\psi=1-w$, we also consider the symbols $W_1,\ldots,W_r$ and $\Psi_2,\ldots,\Psi_r$ formally defined on the one hand by
\begin{equation}\label{25032019E6}
	W_1(\xi) = w\bigg(\frac{\vert Q(B^T)^{r-1}\xi\vert^2}{\langle\xi\rangle^{2\lambda_r/\lambda_{r-1}}}\bigg),
\end{equation}
and on the other hand for all $j\in\{2,\ldots,r\}$ by
\begin{align}
	& W_j(\xi) = w\bigg(\frac{\Gamma^2_j\vert Q(B^T)^{r-j}\xi\vert^2}{\vert Q(B^T)^{r-j+1}\xi\vert^{2\lambda_{r-j+1}/\lambda_{r-j}}}\bigg), \label{25032019E9} \\[5pt]
	& \Psi_j(\xi) = \psi\bigg(\frac{\Gamma^2_j\vert Q(B^T)^{r-j}\xi\vert^2}{\vert Q(B^T)^{r-j+1}\xi\vert^{2\lambda_{r-j+1}/\lambda_{r-j}}}\bigg). \label{27032019E6}
\end{align}
Notice that the function $\psi\in C^{\infty}(\mathbb R,[0,1])$ satisfies
\begin{equation}\label{24072019E4}
	\text{$\psi = 1$ on $\big\{x\in\mathbb R : \vert x\vert\le c_2\big\}$},\quad \Supp\psi\subset\big\{x\in\mathbb R : \vert x\vert\le c_1\big\}.
\end{equation}
We aim at proving by induction that the positive constants $\Gamma_2,\ldots,\Gamma_r\gg1$ can be chosen large enough so that for all $k\in\{1,\ldots,r-1\}$, there exists a smooth  real-valued symbol $p_{k,\varepsilon}\in C^{\infty}(\mathbb R^n)$ satisfying
\begin{equation}\label{23042020E3}
	\forall\xi\in\mathbb R^n,\quad \vert p_{k,\varepsilon}(\xi)\vert\lesssim\langle Q(B^T)^{r-2}\xi\rangle^{q_{r-2}}\langle\xi\rangle^{s_{r-2}},
\end{equation}
and such that for all $\xi\in\mathbb R^n$,
\begin{multline}\label{08042019E18}
	\vert Q(B^T)^{r-k}\xi\vert^{\lambda_{r-k}}\mathcal W_k(\xi)\Psi_{k+1}(\xi) \\[5pt]
	\lesssim\vert Q(B^T)^{r-k-1}\xi\vert^{\lambda_{r-k-1}}\mathcal W_{k+1}(\xi)
	+ \varepsilon\langle\xi\rangle^{\lambda_r} + B^T\xi\cdot\nabla_{\xi}p_{k,\varepsilon}(\xi),
\end{multline}
and
\begin{equation}\label{25032019E2}
	\vert Q(B^T)^{r-k}\xi\vert^{\lambda_{r-k}}\mathcal W_k(\xi)
	\lesssim\vert Q(B^T)^{r-k-1}\xi\vert^{\lambda_{r-k-1}}\mathcal W_{k+1}(\xi)
	+\varepsilon\langle\xi\rangle^{\lambda_r} + B^T\xi\cdot\nabla_{\xi}p_{k,\varepsilon}(\xi),
\end{equation}
where the function $\mathcal W_k\in C^{\infty}(\mathbb R^n)\cap L^{\infty}(\mathbb R^n)$ is given by
\begin{equation}\label{24072019E18}
	\mathcal W_k = \prod_{j=1}^kW_j.
\end{equation} 
Notice from the definition \eqref{23072019E11} of the function $w$ that the symbols $\mathcal W_k$ are well-defined on $\mathbb R^n$ and that for all $2\le j\le k$ and $\xi\in\Supp\mathcal W_k$,
\begin{equation}\label{08042020E4}
	\langle\xi\rangle^{\lambda_r}\lesssim\vert Q(B^T)^{r-1}\xi\vert^{\lambda_{r-1}}\quad\text{and}\quad\vert Q(B^T)^{r-j+1}\xi\vert^{\lambda_{r-j+1}}\lesssim\Gamma_j^{\lambda_{r-j}}\vert Q(B^T)^{r-j}\xi\vert^{\lambda_{r-j}}.
\end{equation}
Synthetically, these estimates can be written in the following way
\begin{multline}\label{08042020E6}
	\langle\xi\rangle^{\lambda_r}\lesssim\vert Q(B^T)^{r-1}\xi\vert^{\lambda_{r-1}}\lesssim_{\ \Gamma_2}\vert Q(B^T)^{r-2}\xi\vert^{\lambda_{r-2}}
	\lesssim_{\ \Gamma_2,\Gamma_3}\ldots \\[5pt]
	\ldots\lesssim_{\ \Gamma_2,\ldots,\Gamma_{k-1}}\vert Q(B^T)^{r-k+1}\xi\vert^{\lambda_{r-k+1}}\lesssim_{\ \Gamma_2,\ldots,\Gamma_k}\vert Q(B^T)^{r-k}\xi\vert^{\lambda_{r-k}}.
\end{multline}
As a consequence, each function $\Psi_{k+1}$ is well-defined on the support of the function $\mathcal W_k$. Notice for futur estimates that from the definition \eqref{24072019E4} the function $\psi$, the following estimate holds for all $1\le k\le r-1$ and $\xi\in\Supp(\mathcal W_k\Psi_{k+1})$,
\begin{equation}\label{08042020E5}
	\Gamma_{k+1}^{\lambda_{r-k-1}}\vert Q(B^T)^{r-k-1}\xi\vert^{\lambda_{r-k-1}}\lesssim\vert Q(B^T)^{r-k}\xi\vert^{\lambda_{r-k}}.
\end{equation}
The estimate \eqref{08042019E18} will be instrumental to prove \eqref{25032019E2} while the estimate \eqref{25032019E2} and a straightforward induction implies that for all $\xi\in\mathbb R^n$,
$$\vert Q(B^T)^{r-1}\xi\vert^{\lambda_{r-1}}\mathcal W_1(\xi)\lesssim\vert Q\xi\vert^{\lambda_0}\mathcal W_r(\xi) + \varepsilon\langle\xi\rangle^{\lambda_r} + \sum_{k=1}^r B^T\xi\cdot\nabla_{\xi}p_{k,\varepsilon}(\xi).$$
This will end the proof of \eqref{23042020E2} and \eqref{23042020E1} since the function $\mathcal W_r$ is bounded by construction and $p_{r,\varepsilon} = \sum_{k=1}^{r-1}p_{k,\varepsilon}$ defines a smooth real-valued symbol $p_{r,\varepsilon}\in C^{\infty}(\mathbb R^n)$ satisfying \eqref{23042020E2} from \eqref{23042020E3}. Notice that we have not used the assumption 
\begin{equation}\label{26062020E2}
	\forall j\in\{2,\ldots,r\},\quad \frac{\lambda_{j-1}}{\lambda_{j-2}}+\frac{\lambda_{j-1}}{\lambda_j}\le 2,
\end{equation}
since the beginning of the proof of Proposition \ref{05122018P1}. It will play its role during this last step.
\newline

\noindent\textit{Basic case.} Let us recall from \eqref{25032019E6}, \eqref{25032019E9} and \eqref{27032019E6} that the symbols $W_1$, $W_2$ and $\Psi_2$ are formally defined for all $\xi\in\mathbb R^n$ by
\begin{align}
	& W_1(\xi) = w\bigg(\frac{\vert Q(B^T)^{r-1}\xi\vert^2}{\langle\xi\rangle^{2\lambda_r/\lambda_{r-1}}}\bigg) \label{23072019E12}, \\[5pt]
	& W_2(\xi) = w\bigg(\frac{\Gamma^2_2\vert Q(B^T)^{r-2}\xi\vert^2}{\vert Q(B^T)^{r-1}\xi\vert^{2\lambda_{r-1}/\lambda_{r-2}}}\bigg), \label{25032019E10}
\end{align}
and
\begin{equation}\label{25032019E7}
	\Psi_2(\xi) = \psi\bigg(\frac{\Gamma^2_2\vert Q(B^T)^{r-2}\xi\vert^2}{\vert Q(B^T)^{r-1}\xi\vert^{2\lambda_{r-1}/\lambda_{r-2}}}\bigg).
\end{equation}
We consider the symbol $p_{1,\varepsilon}\in C^{\infty}(\mathbb R^n)$ defined for all $\xi\in\mathbb R^n$ by
\begin{equation}\label{08042019E5}
	p_{1,\varepsilon}(\xi) = \frac{Q(B^T)^{r-1}\xi\cdot Q(B^T)^{r-2}\xi}{\vert Q(B^T)^{r-1}\xi\vert^{2-\lambda_{r-1}}}\ W_1(\xi)\Psi_2(\xi).
\end{equation}
Notice from \eqref{08042020E6} that the symbol $p_{1,\varepsilon}$ is well-defined on $\mathbb R^n$. Moreover, Cauchy-Schwarz' inequality combined with \eqref{22042020E1} and the estimate \eqref{08042020E5} implies that for all $\xi\in\Supp(W_1\Psi_2)$,
\begin{align}\label{04042019E3}
	\frac{\vert Q(B^T)^{r-1}\xi\cdot Q(B^T)^{r-2}\xi\vert}{\vert Q(B^T)^{r-1}\xi\vert^{2-\lambda_{r-1}}}
	& \lesssim\frac{\vert Q(B^T)^{r-1}\xi\vert\vert Q(B^T)^{r-2}\xi\vert^{1-q_{r-2}+q_{r-2}}}{\vert Q(B^T)^{r-1}\xi\vert^{2-\lambda_{r-1}}} \\
	& \lesssim\langle Q(B^T)^{r-2}\xi\rangle^{q_{r-2}}\frac{\vert Q(B^T)^{r-1}\xi\vert^{(1-q_{r-2})\lambda_{r-1}/\lambda_{r-2}}}{\Gamma_2\vert Q(B^T)^{r-1}\xi\vert^{1-\lambda_{r-1}}} \nonumber \\
	& \lesssim\frac1{\Gamma_2}\langle Q(B^T)^{r-2}\xi\rangle^{q_{r-2}}\langle\xi\rangle^{s_{r-2}}. \nonumber
\end{align}
This estimate implies in particular that the symbol $p_{1,\varepsilon}$ satisfies that for all $\xi\in\mathbb R^n$,
$$\vert p_{1,\varepsilon}(\xi)\vert\lesssim\frac1{\Gamma_2}\langle Q(B^T)^{r-2}\xi\rangle^{q_{r-2}}\langle\xi\rangle^{s_{r-2}}.$$
We want to prove that the positive constant $\Gamma_2\gg_{\varepsilon}1$ can be chosen large enough so that the two following estimates
\begin{equation}\label{08042019E2}
	\vert Q(B^T)^{r-1}\xi\vert^{\lambda_{r-1}}W_1(\xi)\Psi_2(\xi) \\
	\lesssim\vert Q(B^T)^{r-2}\xi\vert^{\lambda_{r-2}}\mathcal W_2(\xi)
	+ \varepsilon\langle\xi\rangle^{\lambda_r} + B^T\xi\cdot\nabla_{\xi}p_{1,\varepsilon}(\xi),
\end{equation}
and
\begin{equation}\label{25072019E3}
	\vert Q(B^T)^{r-1}\xi\vert^{\lambda_{r-1}}W_1(\xi)
	\lesssim\vert Q(B^T)^{r-2}\xi\vert^{\lambda_{r-2}}\mathcal W_2(\xi)
	+\varepsilon\langle\xi\rangle^{\lambda_r} + B^T\xi\cdot\nabla_{\xi}p_{1,\varepsilon}(\xi),
\end{equation}
hold for all $\xi\in\mathbb R^n$. We first focus on proving \eqref{08042019E2}. We therefore need to compute and estimate the quantity $\langle B^T\xi,\nabla_{\xi}\rangle p_{1,\varepsilon}(\xi)$. First, it follows from a direct computation that for all $\xi\in\mathbb R^n$,
$$B^T\xi\cdot\nabla_{\xi}(Q(B^T)^{r-1}\xi\cdot Q(B^T)^{r-2}\xi) = \vert Q(B^T)^{r-1}\xi\vert^2 + Q(B^T)^r\xi\cdot Q(B^T)^{r-2}\xi.$$
We therefore deduce that for all $\xi\in\mathbb R^n$,
$$B^T\xi\cdot\nabla_{\xi}p_{1,\varepsilon}(\xi) = \vert Q(B^T)^{r-1}\xi\vert^{\lambda_{r-1}}W_1(\xi)\Psi_2(\xi) + B_{1,1}(\xi) + B_{2,1}(\xi) + B_{3,1}(\xi) + B_{4,1}(\xi),$$
where the four terms $B_{1,1}(\xi)$, $B_{2,1}(\xi)$, $B_{3,1}(\xi)$ and $B_{4,1}(\xi)$ are respectively defined by
\begin{align*}
	& B_{1,1}(\xi) = \frac{Q(B^T)^r\xi\cdot Q(B^T)^{r-2}\xi}{\vert Q(B^T)^{r-1}\xi\vert^{2-\lambda_{r-1}}}\ W_1(\xi)\Psi_2(\xi), \\[5pt]
	& B_{2,1}(\xi) = \big(B^T\xi\cdot\nabla_{\xi}\vert Q(B^T)^{r-1}\xi\vert^{\lambda_{r-1}-2}\big)(Q(B^T)^{r-1}\xi\cdot Q(B^T)^{r-2}\xi) W_1(\xi)\Psi_2(\xi), \\[5pt]
	& B_{3,1}(\xi) = \frac{Q(B^T)^{r-1}\xi\cdot Q(B^T)^{r-2}\xi}{\vert Q(B^T)^{r-1}\xi\vert^{2-\lambda_{r-1}}}
(B^T\xi\cdot\nabla_{\xi}W_1(\xi))\Psi_2(\xi),
\end{align*}
and 
$$B_{4,1}(\xi) = \frac{Q(B^T)^{r-1}\xi\cdot Q(B^T)^{r-2}\xi}{\vert Q(B^T)^{r-1}\xi\vert^{2-\lambda_{r-1}}}\
W_1(\xi)(B^T\xi\cdot\nabla_{\xi}\Psi_2(\xi)).$$
We now deal with each of these terms one by one. \\[5pt]
\textbf{1.} We deduce from Cauchy-Schwarz' inequality and the estimates \eqref{08042020E4}, \eqref{08042020E5} that the term $B_{1,1}(\xi)$ can be controlled in the following way for all $\xi\in\mathbb R^n$,
$$\vert B_{1,1}(\xi)\vert\lesssim\frac{\vert Q(B^T)^{r-1}\xi\vert^{\lambda_{r-1}/\lambda_r}\vert Q(B^T)^{r-1}\xi\vert^{\lambda_{r-1}/\lambda_{r-2}}}{\Gamma_2\vert Q(B^T)^{r-1}\xi\vert^{2-\lambda_{r-1}}}\ W_1(\xi)\Psi_2(\xi).$$
It therefore follows from the assumption \eqref{26062020E2} that for all $\xi\in\mathbb R^n$,
\begin{equation}\label{27032019E3}
	\vert B_{1,1}(\xi)\vert\lesssim\frac1{\Gamma_2}\vert Q(B^T)^{r-1}\xi\vert^{\lambda_{r-1}}W_1(\xi)\Psi_2(\xi).
\end{equation}
\textbf{2.} A direct computation shows that for all $\xi\in\Supp W_1$,
$$B^T\xi\cdot\nabla_{\xi}\vert Q(B^T)^{r-1}\xi\vert^{\lambda_{r-1}-2}=(\lambda_{r-1}-2)\frac{Q(B^T)^r\xi\cdot Q(B^T)^{r-1}\xi}{\vert Q(B^T)^{r-1}\xi\vert^{4-\lambda_{r-1}}}.$$
The term $B_{2,2}(\xi)$ is thus given from its definition by
$$B_{2,1}(\xi) = (\lambda_{r-1}-2)\frac{Q(B^T)^r\xi\cdot Q(B^T)^{r-1}\xi}{\vert Q(B^T)^{r-1}\xi\vert^2}\frac{Q(B^T)^{r-1}\xi\cdot Q(B^T)^{r-2}\xi}{\vert Q(B^T)^{r-1}\xi\vert^{2-\lambda_{r-1}}}\ W_1(\xi)\Psi_2(\xi).$$
We then deduce from the same ingredients as the one used to estimate the term $B_{1,1}(\xi)$ that for all $\xi\in\mathbb R^n$, the term $B_{2,1}(\xi)$ is itself controlled in the following way for all $\xi\in\mathbb R^n$,
$$\vert B_{2,1}(\xi)\vert\lesssim\frac{\vert Q(B^T)^{r-1}\xi\vert^{\lambda_{r-1}/\lambda_r}\vert Q(B^T)^{r-1}\xi\vert}{\vert Q(B^T)^{r-1}\xi\vert^2}
\frac{\vert Q(B^T)^{r-1}\xi\vert\vert Q(B^T)^{r-1}\xi\vert^{\lambda_{r-1}/\lambda_{r-2}}}{\Gamma_2\vert Q(B^T)^{r-1}\xi\vert^{2-\lambda_{r-1}}}W_1(\xi)\Psi_2(\xi).$$
Using \eqref{26062020E2} anew, we deduce that for all $\xi\in\mathbb R^n$,
\begin{equation}\label{26032019E2}
	\vert B_{2,1}(\xi)\vert\lesssim\frac1{\Gamma_2}\vert Q(B^T)^{r-1}\xi\vert^{\lambda_{r-1}}W_1(\xi)\Psi_2(\xi).
\end{equation}
\textbf{3.} Notice that Cauchy-Schwarz' inequality combined with the estimate \eqref{08042020E5} implies that for all $\xi\in\Supp(W_1\Psi_2)$,
\begin{equation}\label{24042020E4}
	\frac{\vert Q(B^T)^{r-1}\xi\cdot Q(B^T)^{r-2}\xi\vert}{\vert Q(B^T)^{r-1}\xi\vert^{2-\lambda_{r-1}}}\lesssim\frac{\vert Q(B^T)^{r-1}\xi\vert^{1+\lambda_{r-1}/\lambda_{r-2}}}{\Gamma_2\vert Q(B^T)^{r-1}\xi\vert^{2-\lambda_{r-1}}}.
\end{equation}
The term $B_{3,1}(\xi)$ can be directly controlled by using this estimate, the assumption \eqref{26062020E2} and Lemma \ref{29032019L1}, the function $w$ being a smooth function satisfying that $w'$ is compactly supported with $\max\Supp w' = c_1>0$, the condition $\lambda_{r-1}<\lambda_r$ being assumed. We obtain that for all $\xi\in\mathbb R^n$,
\begin{align}\label{26032019E3}
	\vert B_{3,1}(\xi)\vert & \lesssim\frac{\vert Q(B^T)^{r-1}\xi\vert^{1+\lambda_{r-1}/\lambda_{r-2}}}{\Gamma_2\vert Q(B^T)^{r-1}\xi\vert^{2-\lambda_{r-1}}}
	\langle\xi\rangle^{1-\lambda_r/\lambda_{r-1}}\bigg\vert w'\bigg(\frac{\vert Q(B^T)^{r-1}\xi\vert^2}{\langle\xi\rangle^{2\lambda_r/\lambda_{r-1}}}\bigg)\bigg\vert\Psi_2(\xi) \\[5pt]
	& \lesssim\frac{\vert Q(B^T)^{r-1}\xi\vert^{1+\lambda_{r-1}/\lambda_{r-2}}}{\Gamma_2\vert Q(B^T)^{r-1}\xi\vert^{2-\lambda_{r-1}}}
	\frac{\vert Q(B^T)^{r-1}\xi\vert^{\lambda_{r-1}/\lambda_r}}{\vert Q(B^T)^{r-1}\xi\vert}\bigg\vert w'\bigg(\frac{\vert Q(B^T)^{r-1}\xi\vert^2}{\langle\xi\rangle^{2\lambda_r/\lambda_{r-1}}}\bigg)\bigg\vert\Psi_2(\xi) \nonumber \\[5pt]
	& \lesssim\frac1{\Gamma_2}\vert Q(B^T)^{r-1}\xi\vert^{\lambda_{r-1}}\bigg\vert w'\bigg(\frac{\vert Q(B^T)^{r-1}\xi\vert^2}{\langle\xi\rangle^{2\lambda_r/\lambda_{r-1}}}\bigg)\bigg\vert\Psi_2(\xi)\lesssim\frac1{\Gamma_2}\langle\xi\rangle^{\lambda_r}. \nonumber
\end{align}
\textbf{4.} The estimate \eqref{24042020E4} and the proof of Lemma \ref{25072019L1} directly imply that for all $\xi\in\mathbb R^n$,
\begin{multline*}
	\vert B_{4,1}(\xi)\vert\lesssim_{\ \Gamma_2}\frac{\vert Q(B^T)^{r-1}\xi\vert^{1+\lambda_{r-1}/\lambda_{r-2}}}{\vert Q(B^T)^{r-1}\xi\vert^{2-\lambda_{r-1}}}
	\bigg(\frac{\vert Q(B^T)^{r-1}\xi\vert}{\vert Q(B^T)^{r-1}\xi\vert^{\lambda_{r-1}/\lambda_{r-2}}} \\
	+ \frac{\vert Q(B^T)^{r-1}\xi\vert^{\lambda_{r-1}/\lambda_r}}{\vert Q(B^T)^{r-1}\xi\vert}\bigg)\mathcal W_2(\xi).
\end{multline*}
Using the assumption \eqref{26062020E2} and \eqref{08042020E4}, we deduce that there exists a positive constant $\Lambda_{\Gamma_2}>0$ depending only on $\Gamma_2\gg1$ such that for all $\xi\in\mathbb R^n$,
\begin{equation}\label{27032019E1}
	\vert B_{4,1}(\xi)\vert \lesssim\Lambda_{\Gamma_2}\vert Q(B^T)^{r-2}\xi\vert^{\lambda_{r-2}}\mathcal W_2(\xi).
\end{equation}
Gathering the estimates \eqref{27032019E3}, \eqref{26032019E2}, \eqref{26032019E3} and \eqref{27032019E1}, we deduce that there exists a positive constant $c>0$ independent of $\varepsilon>0$ and $\Gamma_2,\ldots,\Gamma_r\gg1$ such that for all $\xi\in\mathbb R^n$,
\begin{multline*}
	\Big(1-\frac c{\Gamma_2}\Big)\vert Q(B^T)^{r-1}\xi\vert^{\lambda_{r-1}}W_1(\xi)\Psi_2(\xi) \\
	\lesssim\Lambda_{\Gamma_2}\vert Q(B^T)^{r-2}\xi\vert^{\lambda_{r-2}}\mathcal W_2(\xi) + \frac1{\Gamma_2}\langle\xi\rangle^{\lambda_r} + B^T\xi\cdot\nabla_{\xi}p_{1,\varepsilon}(\xi).
\end{multline*}
We adjust the large constant $\Gamma_2\gg_{\varepsilon}1$ so that the inequality \eqref{08042019E2} holds. The value of the positive constant $\Gamma_2\gg1$ is now fixed and will not be explicit anymore in the following. To end the basic case, we check that the estimate \eqref{25072019E3} can be deduced from \eqref{08042019E2}. By definition, the functions $w$ and $\psi$ satisfy $1=w+\psi$, and as a consequence, the following equality holds
\begin{equation}\label{24072019E1}
	W_1 = W_1W_2 + W_1\Psi_2 = \mathcal W_2 + W_1\Psi_2.
\end{equation}
Moreover, we notice from \eqref{08042020E4} that for all $\xi\in\mathbb R^n$,
$$\vert Q(B^T)^{r-1}\xi\vert^{\lambda_{r-1}}\mathcal W_2(\xi)\lesssim\vert Q(B^T)^{r-2}\xi\vert^{\lambda_{r-2}}\mathcal W_2(\xi).$$
As a consequence of \eqref{24072019E1} and the above estimate, the following inequality holds for all $\xi\in\mathbb R^n$,
$$\vert Q(B^T)^{r-1}\xi\vert^{\lambda_{r-1}}W_1(\xi)\lesssim\vert Q(B^T)^{r-2}\xi\vert^{\lambda_{r-2}}\mathcal W_2(\xi)
+ \vert Q(B^T)^{r-1}\xi\vert^{\lambda_{r-1}}W_1(\xi)\Psi_2(\xi).$$
The estimate \eqref{25072019E3} is then a consequence of this inequality and \eqref{08042019E2}. This proves the induction hypothesis in the basic case. 
\newline

\noindent\textit{Induction.} Let $2\le k\le r-1$. We assume that the values of the large positive constants $\Gamma_2,\ldots,\Gamma_k\gg_{\varepsilon}1$ have already been chosen and that we have constructed some smooth real-valued symbols $p_{1,\varepsilon},\ldots,p_{k-1,\varepsilon}\in C^{\infty}(\mathbb R^n)$ satisfying that for all $1\le j\le k-1$ 
\begin{equation}\label{24042020E3}
	\forall\xi\in\mathbb R^n,\quad \vert p_{j,\varepsilon}(\xi)\vert\lesssim\langle Q(B^T)^{r-2}\xi\rangle^{q_{r-2}}\langle\xi\rangle^{s_{r-2}},
\end{equation}
and such that for all $\xi\in\mathbb R^n$,
\begin{multline}\label{08042019E11}
	\vert Q(B^T)^{r-j}\xi\vert^{\lambda_{r-j}}\mathcal W_j(\xi)\Psi_{j+1}(\xi) \\[5pt]
	\lesssim\vert Q(B^T)^{r-j-1}\xi\vert^{\lambda_{r-j-1}}\mathcal W_{j+1}(\xi) + \varepsilon\langle\xi\rangle^{\lambda_r} + B^T\xi\cdot\nabla_{\xi}p_{j,\varepsilon}(\xi),
\end{multline}
and
\begin{equation}\label{25072019E4}
	\vert Q(B^T)^{r-j}\xi\vert^{\lambda_{r-j}}\mathcal W_j(\xi)
	\lesssim\vert Q(B^T)^{r-j-1}\xi\vert^{\lambda_{r-j-1}}\mathcal W_{j+1}(\xi)
	+\varepsilon\langle\xi\rangle^{\lambda_r} + B^T\xi\cdot\nabla_{\xi}p_{j,\varepsilon}(\xi).
\end{equation}
Notice that the estimate \eqref{25072019E4} and a straightforward induction implies in particular that for all $1\le j\le k-1$ and $\xi\in\mathbb R^n$,
\begin{equation}\label{25072019E6}
	\vert Q(B^T)^{r-j}\xi\vert^{\lambda_{r-j}}\mathcal W_j(\xi)\lesssim\vert Q(B^T)^{r-k}\xi\vert^{\lambda_{r-k}}\mathcal W_k(\xi)
	+\varepsilon\langle\xi\rangle^{\lambda_r} + \sum_{l=j}^{k-1}B^T\xi\cdot\nabla_{\xi}p_{l,\varepsilon}(\xi).
\end{equation}
Since the values of the large constants $\Gamma_2,\ldots,\Gamma_k\gg_{\varepsilon}1$ have already been chosen, we will not make them appear in the futur estimates. Recalling from \eqref{25032019E6}, \eqref{25032019E9}, \eqref{27032019E6} and \eqref{24072019E18} that $\mathcal W_{k+1}(\xi)$ and $\Psi_{k+1}(\xi)$ are given for all $\xi\in\mathbb R^n$ by
\begin{equation}\label{24072019E14}
	\mathcal W_k(\xi) = \prod_{j=1}^kW_j(\xi) = w\bigg(\frac{\vert Q(B^T)^{r-1}\xi\vert^2}{\langle\xi\rangle^{2\lambda_r/\lambda_{r-1}}}\bigg)\prod_{j=2}^kw\bigg(\frac{\Gamma^2_j\vert Q(B^T)^{r-j}\xi\vert^2}{\vert Q(B^T)^{r-j+1}\xi\vert^{2\lambda_{r-j+1}/\lambda_{r-j}}}\bigg),
\end{equation}
and
\begin{equation}\label{24072019E13}
	\Psi_{k+1}(\xi) = \psi\bigg(\frac{\Gamma^2_{k+1}\vert Q(B^T)^{r-k-1}\xi\vert^2}{\vert Q(B^T)^{r-k}\xi\vert^{2\lambda_{r-k}/\lambda_{r-k+1}}}\bigg),
\end{equation}
we want to adjust the value of the positive constant $\Gamma_{k+1}\gg_{\varepsilon}1$ and construct a smooth real-valued symbol $p_{k,\varepsilon}\in C^{\infty}(\mathbb R^n)$ satisfying
\begin{equation}\label{24042020E1}
	\forall\xi\in\mathbb R^n,\quad \vert p_{k,\varepsilon}(\xi)\vert\lesssim\langle Q(B^T)^{r-2}\xi\rangle^{q_{r-2}}\langle\xi\rangle^{s_{r-2}},
\end{equation}
and such that for all $\xi\in\mathbb R^n$,
\begin{multline}\label{08042019E10}
	\vert Q(B^T)^{r-k}\xi\vert^{\lambda_{r-k}}\mathcal W_k(\xi)\Psi_{k+1}(\xi)
	\lesssim\vert Q(B^T)^{r-k-1}\xi\vert^{\lambda_{r-k-1}}\mathcal W_{k+1}(\xi) \\[5pt]
	+ \varepsilon\langle\xi\rangle^{\lambda_r} + B^T\xi\cdot\nabla_{\xi}p_{k,\varepsilon}(\xi),
\end{multline}
and
\begin{equation}\label{25072019E5}
	\vert Q(B^T)^{r-k}\xi\vert^{\lambda_{r-k}}\mathcal W_k(\xi)
	\lesssim\vert Q(B^T)^{r-k-1}\xi\vert^{\lambda_{r-k-1}}\mathcal W_{k+1}(\xi)
	+\varepsilon\langle\xi\rangle^{\lambda_r} + B^T\xi\cdot\nabla_{\xi}p_{k,\varepsilon}(\xi).
\end{equation}
We first focus on proving the estimate \eqref{08042019E10}. To that end, we consider the smooth symbol $\mathfrak p_{k,\varepsilon}\in C^{\infty}(\mathbb R^n)$ defined for all $\xi\in\mathbb R^n$ by
\begin{equation}\label{24072019E12}
	\mathfrak p_{k,\varepsilon}(\xi) = \frac{Q(B^T)^{r-k}\xi\cdot Q(B^T)^{r-k-1}\xi}{\vert Q(B^T)^{r-k}\xi\vert^{2-\lambda_{r-k}}}\ \mathcal W_k(\xi)\Psi_{k+1}(\xi).
\end{equation}
Notice that the symbol $\mathfrak p_{k,\varepsilon}$ is well-defined from \eqref{08042020E4}. Moreover, Cauchy-Schwarz' inequality and \eqref{08042020E5} imply that for all $\xi\in\Supp(\mathcal W_k\Psi_{k+1})$,
\begin{equation}\label{24072019E5}
	\frac{\vert Q(B^T)^{r-k}\xi\cdot Q(B^T)^{r-k-1}\xi\vert}{\vert Q(B^T)^{r-k}\xi\vert^{2-\lambda_{r-k}}}
	\lesssim\frac{\vert Q(B^T)^{r-k}\xi\vert\vert Q(B^T)^{r-k}\xi\vert^{\lambda_{r-k}/\lambda_{r-k-1}}}{\Gamma_{k+1}\vert Q(B^T)^{r-k}\xi\vert^{2-\lambda_{r-k}}}=\frac1{\Gamma_{k+1}}.
\end{equation}
This estimate with \eqref{24072019E12} shows in particular that
\begin{equation}\label{24042020E2}
	\forall\xi\in\mathbb R^n,\quad\vert\mathfrak p_{k,\varepsilon}(\xi)\vert\lesssim\frac1{\Gamma_{k+1}}.
\end{equation}
In order to study the quantity $B^T\xi\cdot\nabla_{\xi}\mathfrak p_{k,\varepsilon}$, we compute that for all $\xi\in\mathbb R^n$,
$$B^T\xi\cdot\nabla_{\xi}(Q(B^T)^{r-k}\xi\cdot Q(B^T)^{r-k-1}\xi) = \vert Q(B^T)^{r-k}\xi\vert^2 + Q(B^T)^{r-k+1}\xi\cdot Q(B^T)^{r-k-1}\xi.$$
We therefore deduce that for all $\xi\in\mathbb R^n$,
\begin{multline*}
	B^T\xi\cdot\nabla_{\xi}\mathfrak p_{k,\varepsilon}(\xi) = \vert Q(B^T)^{r-k}\xi\vert^{\lambda_{r-k}}\mathcal W_k(\xi)\Psi_{k+1}(\xi) \\[5pt]
	+ B_{1,k}(\xi) + B_{2,k}(\xi) + B_{3,k}(\xi) + B_{4,k}(\xi) + B_{5,k}(\xi),
\end{multline*}
where the five terms $B_{1,k}(\xi)$, $B_{2,k}(\xi)$, $B_{3,k}(\xi)$, $B_{4,k}(\xi)$ and $B_{5,k}(\xi)$ are given by
\begin{align*}
	& B_{1,k}(\xi) = \frac{Q(B^T)^{r-k+1}\xi\cdot Q(B^T)^{r-k-1}\xi}{\vert Q(B^T)^{r-k}\xi\vert^{2-\lambda_{r-k}}}\ \mathcal W_k(\xi)\Psi_{k+1}(\xi), \\[5pt]
	& B_{2,k}(\xi) = (Q(B^T)^{r-k}\xi\cdot Q(B^T)^{r-k-1}\xi)\big(B^T\xi\cdot\nabla_{\xi}\vert Q(B^T)^{r-k}\xi\vert^{\lambda_{r-k}-2}\big)\mathcal W_k(\xi)\Psi_{k+1}(\xi), \\[5pt]
	& B_{3,k}(\xi) = \frac{Q(B^T)^{r-k}\xi\cdot Q(B^T)^{r-k-1}\xi}{\vert Q(B^T)^{r-k}\xi\vert^{2-\lambda_{k-1}}}
(B^T\xi\cdot\nabla_{\xi}W_1(\xi))\bigg(\prod_{j=2}^kW_j(\xi)\bigg)\Psi_{k+1}(\xi), \\[5pt]
	& B_{4,k}(\xi) = \frac{Q(B^T)^{r-k}\xi\cdot Q(B^T)^{r-k-1}\xi}{\vert Q(B^T)^{r-k}\xi\vert^{2-\lambda_{r-k}}}\
\mathcal W_k(\xi)(B^T\xi\cdot\nabla_{\xi}\Psi_{k+1}(\xi)),
\end{align*}
and
$$B_{5,k}(\xi) = \frac{Q(B^T)^{r-k}\xi\cdot Q(B^T)^{r-k-1}\xi}{\vert Q(B^T)^{r-k}\xi\vert^{2-\lambda_{r-k}}}\
W_1(\xi)\ B^T\xi\cdot\nabla_{\xi}\bigg(\prod_{j=2}^kW_j(\xi)\bigg)\Psi_{k+1}(\xi).$$
As we have done in the basic case when $k=1$, we now consider each term one by one. The strategy used to estimate the four ones follows almost line to line the one used to control the terms $B_{1,1}(\xi)$, $B_{2,1}(\xi)$, $B_{3,1}(\xi)$ and $B_{4,1}(\xi)$ in the basic case. \\[5pt]
\textbf{1.} We deduce from Cauchy-Schwarz' inequality and the estimates \eqref{08042020E4} and \eqref{08042020E5} that for all $\xi\in\mathbb R^n$,
$$\vert B_{1,k}(\xi)\vert\lesssim\frac{\vert Q(B^T)^{r-k}\xi\vert^{\lambda_{r-k}/\lambda_{r-k+1}}\vert Q(B^T)^{r-k}\xi\vert^{\lambda_{r-k}/\lambda_{r-k-1}}}{\Gamma_{k+1}\vert Q(B^T)^{r-k}\xi\vert^{2-\lambda_{k-1}}}\mathcal W_k(\xi)\Psi_{k+1}(\xi).$$
It therefore follows from the assumption \eqref{26062020E2} that for all $\xi\in\mathbb R^n$,
\begin{equation}\label{27032019E11}
	\vert B_{1,k}(\xi)\vert\lesssim\frac1{\Gamma_{k+1}}\vert Q(B^T)^{r-k}\xi\vert^{\lambda_{r-k}}\mathcal W_k(\xi)\Psi_{k+1}(\xi).	
\end{equation}
\textbf{2.} A direct computation shows that for all $\xi\in\Supp \mathcal W_k$,
$$B^T\xi\cdot\nabla_{\xi}\vert Q(B^T)^{r-k}\xi\vert^{\lambda_{r-k}-2}=(\lambda_{r-k}-2)\frac{Q(B^T)^{r-k+1}\xi\cdot Q(B^T)^{r-k}\xi}{\vert Q(B^T)^{r-k}\xi\vert^{4-\lambda_{r-k}}}.$$
Consequently, the term $B_{2,k}(\xi)$ is given for all $\xi\in\mathbb R^n$ from its definition by
\begin{multline*}
	B_{2,k}(\xi) = (\lambda_{r-k}-2)\frac{Q(B^T)^{r-k+1}\xi\cdot Q(B^T)^{r-k}\xi}{\vert Q(B^T)^{r-k}\xi\vert^2} \\
	\times\frac{Q(B^T)^{r-k}\xi\cdot Q(B^T)^{r-k-1}\xi}{\vert Q(B^T)^{r-k}\xi\vert^{2-\lambda_{r-k}}}\ \mathcal W_k(\xi)\Psi_{k+1}(\xi).
\end{multline*}
By using the same ingredients as the ones used to estimate the term $B_{1,k}(\xi)$, we deduce that for all $\xi\in\mathbb R^n$,
\begin{multline*}
	\vert B_{2,k}(\xi)\vert\lesssim\frac{\vert Q(B^T)^{r-k}\xi\vert^{\lambda_{r-k}/\lambda_{r-k+1}}\vert Q(B^T)^{r-k}\xi\vert}{\vert Q(B^T)^{r-k}\xi\vert^2} \\
	\times\frac{\vert Q(B^T)^{r-k}\xi\vert\vert Q(B^T)^{r-k}\xi\vert^{\lambda_{r-k}/\lambda_{r-k-1}}}{\Gamma_{k+1}\vert Q(B^T)^{r-k}\xi\vert^{2-\lambda_{r-k}}}\ \mathcal W_k(\xi)\Psi_{k+1}(\xi).
\end{multline*}
The assumption \eqref{26062020E2} then implies that for all $\xi\in\mathbb R^n$,
\begin{equation}\label{27032019E12}
	\vert B_{2,k}(\xi)\vert\lesssim\frac1{\Gamma_{k+1}}\vert Q(B^T)^{r-k}\xi\vert^{\lambda_{r-k}}\mathcal W_k(\xi)\Psi_{k+1}(\xi).
\end{equation}
\textbf{3.} 
The term $B_{3,k}(\xi)$ can be controlled directly by using \eqref{24072019E5} and Lemma \ref{29032019L1}, the function $w$ being a smooth function such that $w'$ is compactly supported with $\max\Supp w' = c_1>0$, the condition $\lambda_r<\lambda_{r-1}$ being assumed. Using these ingredients, we obtain that for all $\xi\in\mathbb R^n$,
$$\vert B_{3,k}(\xi)\vert\lesssim\frac1{\Gamma_{k+1}}
\langle\xi\rangle^{1-\lambda_r/\lambda_{r-1}}\bigg\vert w'\bigg(\frac{\vert Q(B^T)^{r-1}\xi\vert^2}{\langle\xi\rangle^{2\lambda_r/\lambda_{r-1}}}\bigg)\bigg\vert
\bigg(\prod_{j=2}^kW_j(\xi)\bigg)\Psi_{k+1}(\xi).$$
Moreover, we deduce from \eqref{22042020E1} that 
$$1-\frac{\lambda_r}{\lambda_{r-1}} = \lambda_r-s_{r-1}-\frac{q_{r-1}\lambda_r}{\lambda_{r-1}}\le\lambda_r,$$
Exploiting this estimate and the fact that the functions $w'$ and $\psi$ are bounded, we deduce that for all $\xi\in\mathbb R^n$,
\begin{equation}\label{29032019E3}
	\vert B_{3,k}(\xi)\vert\lesssim\frac1{\Gamma_{k+1}}\langle\xi\rangle^{\lambda_r}.
\end{equation}
\textbf{4.} We deduce from the estimate \eqref{24072019E5} and Lemma \ref{25072019L1} that there exists a positive constant $\Lambda_{\Gamma_k,\Gamma_{k+1}}>0$ depending on $\Gamma_k,\Gamma_{k+1}$ such that the term $B_{4,k}(\xi)$ can be controlled in the following way for all $\xi\in\mathbb R^n$,
\begin{equation}\label{24072019E17}
	\vert B_{4,k}(\xi)\vert\lesssim\Lambda_{\Gamma_k, \Gamma_{k+1}}\vert Q(B^T)^{r-k-1}\xi\vert^{\lambda_{r-k-1}}\mathcal W_{k+1}(\xi).
\end{equation}
Since the value of the constant $\Gamma_k$ has already been fixed, the dependence of $\Lambda_{\Gamma_k,\Gamma_{k+1}}$ with respect to this constant will not be explicit in the following, and we will simply be denoted $\Lambda_{\Gamma_{k+1}}$ instead of $\Lambda_{\Gamma_k,\Gamma_{k+1}}$. \\[5pt]
\textbf{5.} Finally, we consider the term $B_{5,k}(\xi)$, which we recall is given for all $\xi\in\mathbb R^n$ by
$$B_{5,k}(\xi) = \frac{Q(B^T)^{r-k}\xi\cdot Q(B^T)^{r-k-1}\xi}{\vert Q(B^T)^{r-k}\xi\vert^{2-\lambda_{r-k}}}\
W_1(\xi)\ B^T\xi\cdot\nabla_{\xi}\bigg(\prod_{j=2}^kW_j(\xi)\bigg)\Psi_{k+1}(\xi).$$
As a consequence of Leibniz' formula, this term splits in the following way for all $\xi\in\mathbb R^n$,
\begin{equation}\label{25072019E1}
	B_{5,k}(\xi) = \sum_{j=2}^kB_{5,k,j}(\xi),
\end{equation}
where the quantities $B_{5,k,j}(\xi)$ are defined for all $2\le j\le k$ and $\xi\in\mathbb R^n$ by
$$B_{5,k,j}(\xi) = \frac{Q(B^T)^{r-k}\xi\cdot Q(B^T)^{r-k-1}\xi}{\vert Q(B^T)^{r-k}\xi\vert^{2-\lambda_{r-k}}}\ \mathcal W_{j-1}(\xi)\big(B^T\xi\cdot\nabla_{\xi}W_j(\xi)\big)\bigg(\prod_{l=j+1}^kW_l(\xi)\bigg)\Psi_{k+1}(\xi).$$
We deduce from \eqref{08042020E5}, \eqref{24072019E5} and Lemma \ref{25072019L1} that for all $j\in\{2,\ldots,k\}$, there exists a positive constant $\Lambda_{\Gamma_{j-1},\Gamma_j}>0$ only depending on $\Gamma_{j-1}$ and $\Gamma_j$, and only on $\Gamma_2$ when $j=2$, such that for all $\xi\in\mathbb R^n$,
\begin{align*}
	\vert B_{5,k,j}(\xi)\vert & \lesssim\frac{\Lambda_{\Gamma_{j-1},\Gamma_j}}{\Gamma_{k+1}}\vert Q(B^T)^{r-j}\xi\vert^{\lambda_{r-j}}\mathcal W_{j-1}(\xi)\Psi_j(\xi) \\[5pt]
	& \lesssim\frac{\Lambda_{\Gamma_{j-1},\Gamma_j}}{\Gamma_{k+1}}\vert Q(B^T)^{r-j+1}\xi\vert^{\lambda_{r-j+1}}\mathcal W_{j-1}(\xi)\Psi_j(\xi). \nonumber
\end{align*}
Since the values of the constants $\Gamma_2,\ldots,\Gamma_k$ have already been chosen, we will not make the constants $\Lambda_{\Gamma_{j-1},\Gamma_j}$ appear in the following. We deduce from \eqref{25072019E1} and the above estimate that for all $\xi\in\mathbb R^n$,
\begin{align*}
	\vert B_{5,k}(\xi)\vert & \lesssim\frac1{\Gamma_{k+1}}\sum_{j=2}^k\vert Q(B^T)^{r-j+1}\xi\vert^{\lambda_{r-j+1}}\mathcal W_{j-1}(\xi)\Psi_j(\xi) \\
	& = \frac1{\Gamma_{k+1}}\sum_{j=1}^{k-1}\vert Q(B^T)^{r-j}\xi\vert^{\lambda_{r-j}}\mathcal W_j(\xi)\Psi_{j+1}(\xi) \nonumber.
\end{align*}
Then, the induction hypothesis \eqref{08042019E11} implies that for all $\xi\in\mathbb R^n$,
\begin{equation}\label{25072019E8}
	\vert B_{5,k}(\xi)\vert \lesssim \frac1{\Gamma_{k+1}}\sum_{j=1}^{k-1}\Big(\vert Q(B^T)^{r-j-1}\xi\vert^{\lambda_{r-j-1}}\mathcal W_{j+1}(\xi)
	+ \varepsilon\langle\xi\rangle^{\lambda_r} + B^T\xi\cdot\nabla_{\xi}p_{j,\varepsilon}(\xi)\Big).
\end{equation}
We also deduce from \eqref{25072019E6} that the previous sum can be controlled for all $\xi\in\mathbb R^n$ in the following way:
\begin{align}\label{25072019E7}
	&\ \sum_{j=1}^{k-1}\vert Q(B^T)^{r-j-1}\xi\vert^{\lambda_{r-j-1}}\mathcal W_{j+1}(\xi) = \sum_{j=2}^k\vert Q(B^T)^{r-j}\xi\vert^{\lambda_{r-j}}\mathcal W_j(\xi) \\
	= &\ \vert Q(B^T)^{r-k}\xi\vert^{\lambda_{r-k}}\mathcal W_k(\xi) + \sum_{j=2}^{k+1}\vert Q(B^T)^{r-j}\xi\vert^{\lambda_{r-j}}\mathcal W_j(\xi) \nonumber \\
	\lesssim &\ \vert Q(B^T)^{r-k}\xi\vert^{\lambda_{r-k}}\mathcal W_k(\xi) + \varepsilon\langle\xi\rangle^{\lambda_r} + \sum_{j=2}^{k-1}\sum_{l=j}^{k-1}B^T\xi\cdot\nabla_{\xi}p_{l,\varepsilon}(\xi). \nonumber
\end{align}
Combining \eqref{25072019E8} and \eqref{25072019E7} and using that $\Gamma_{k+1}\gg1$, we obtain that for all $\xi\in\mathbb R^n$,
\begin{equation}\label{25072019E9}
	\vert B_{5,k}(\xi)\vert\lesssim\frac1{\Gamma_{k+1}}\vert Q(B^T)^{r-k}\xi\vert^{\lambda_{r-k}}\mathcal W_k(\xi) + \varepsilon\langle\xi\rangle^{\lambda_r}
	+ \sum_{j=1}^{k-1}B^T\xi\cdot\nabla_{\xi}p_{j,\varepsilon}(\xi).
\end{equation}
We deduce from \eqref{27032019E11}, \eqref{27032019E12}, \eqref{29032019E3}, \eqref{24072019E17} and \eqref{25072019E9} that there exists a positive constant $c>0$ independent of $\varepsilon>0$ and $\Gamma_2,\ldots,\Gamma_r\gg1$ such that for all $\xi\in\mathbb R^n$,
\begin{multline}\label{08042019E16}
	\Big(1-\frac c{\Gamma_{k+1}}\Big)\vert Q(B^T)^{r-k}\xi\vert^{\lambda_{r-k}}\mathcal W_k(\xi)\Psi_{k+1}(\xi) 
	\lesssim \frac1{\Gamma_{k+1}}\vert Q(B^T)^{r-k}\xi\vert^{\lambda_{r-k}}\mathcal W_k(\xi) \\[5pt]
	+ \Lambda_{\Gamma_{k+1}}\vert Q(B^T)^{r-k-1}\xi\vert^{\lambda_{r-k-1}}\mathcal W_{k+1}(\xi) 
	+ \frac1{\Gamma_{k+1}}\langle\xi\rangle^{\lambda_r} + \varepsilon\langle\xi\rangle^{\lambda_r}
	+ B^T\xi\cdot\nabla_{\xi}p_{k,\varepsilon}(\xi),
\end{multline}
where we defined the symbol $p_{k,\varepsilon}\in C^{\infty}(\mathbb R^n)$ by $p_{k,\varepsilon} = \mathfrak p_{k,\varepsilon} + \sum_{j=1}^{k-1}p_{j,\varepsilon}$. From there, we can prove that the constant $\Gamma_{k+1}\gg_{\varepsilon}1$ can be chosen large enough so that the estimate \eqref{25072019E5} holds.
By definition, the functions $w$ and $\psi$ satisfy $1 = w+\psi$, and as a consequence, the following equality holds
\begin{equation}\label{25072019E10}
	\mathcal W_k = \mathcal W_kW_{k+1} + \mathcal W_k\Psi_{k+1} = \mathcal W_{k+1} + \mathcal W_k\Psi_{k+1}.
\end{equation}
Moreover, we get from \eqref{08042020E4} that for all $\xi\in\mathbb R^n$,
\begin{equation}\label{25072019E11}
	\vert Q(B^T)^{r-k}\xi\vert^{\lambda_{r-k}}\mathcal W_{k+1}(\xi)\lesssim\Gamma_{k+1}\vert Q(B^T)^{r-k-1}\xi\vert^{\lambda_{r-k-1}}\mathcal W_{k+1}(\xi).
\end{equation}
As a consequence of \eqref{25072019E10} and \eqref{25072019E11}, the following inequality holds for all $\xi\in\mathbb R^n$,
\begin{multline*}
	\vert Q(B^T)^{r-k}\xi\vert^{\lambda_{r-k}}\mathcal W_k(\xi) \\[5pt]
	\lesssim\Gamma_{k+1}\vert Q(B^T)^{r-k-1}\xi\vert^{\lambda_{r-k-1}}\mathcal W_{k+1}(\xi)
	+ \vert Q(B^T)^{r-k}\xi\vert^{\lambda_{r-k}}\mathcal W_k(\xi)\Psi_{k+1}(\xi).
\end{multline*}
Choosing $\Gamma_{k+1}>c$, we deduce from \eqref{08042019E16} and the above estimate that for all $\xi\in\mathbb R^n$,
\begin{multline*}
	\vert Q(B^T)^{r-k}\xi\vert^{\lambda_{r-k}}\mathcal W_k(\xi)\lesssim\frac1{\Gamma_{k+1}}\vert Q(B^T)^{r-k}\xi\vert^{\lambda_{r-k}}\mathcal W_k(\xi) \\[5pt]
	+\big(\Lambda_{\Gamma_{k+1}}+\Gamma_{k+1}\big)\vert Q(B^T)^{r-k-1}\xi\vert^{\lambda_{r-k-1}}\mathcal W_{k+1}(\xi) \\[5pt]
	+\frac1{\Gamma_{k+1}}\langle\xi\rangle^{\lambda_r}+\varepsilon\langle\xi\rangle^{\lambda_r} + B^T\xi\cdot\nabla_{\xi}p_{k,\varepsilon}(\xi).
\end{multline*}
By choosing the constant $\Gamma_{k+1}\gg_{\varepsilon}1$ large enough, we deduce that the estimate \eqref{25072019E5} holds. Then, by combining \eqref{25072019E5} and \eqref{08042019E16} and choosing $\Gamma_{k+1}\gg1$ even more larger if necessary, we deduce that the estimate \eqref{08042019E10} also holds. The value of the constant $\Gamma_{k+1}\gg1$ is now fixed. Notice from \eqref{24042020E3} and \eqref{24042020E2} that $p_{k,\varepsilon}$ satisfies the estimate \eqref{24042020E1}. This ends the induction step and the proof of \eqref{23042020E2} and \eqref{23042020E1}.
\end{proof}

\section{Appendix}\label{appendix}

\subsection{About the Kalman rank condition} Let us begin this appendix by giving a manageable characterization of the Kalman rank condition.

\begin{lem}\label{29082018E1} Let $B$ and $Q$ be real $n\times n$ matrices, with $Q$ symmetric positive semidefinite. The Kalman rank condition \eqref{10052018E4} holds if and only if
$$\bigcap_{j=0}^{n-1}\Ker\big(Q(B^T)^j\big) = \{0\}.$$
\end{lem}

\begin{proof} Using the notations of \eqref{10052018E4}, we have the following equivalences:
\begin{align*}
	\Rank\big[B\ \vert\ Q\big] = n & \Leftrightarrow \Ran\big[B\ \vert\ Q\big] = \mathbb R^n, \\[5pt]
	& \Leftrightarrow \Ker\Big(\big[B\ \vert\ Q\big]^T\Big) = \Big(\Ran\big[B\ \vert\ Q\big]\Big)^{\perp} = \{0\}, \\[5pt]
	& \Leftrightarrow \bigcap_{j=0}^{n-1}\Ker\big(Q(B^T)^j\big) = \{0\},
\end{align*}
where $\perp$ denotes the orthogonality with respect to the canonical Euclidean structure.
\end{proof}

\subsection{A lemma} The following lemma is the core of the multiplier method used in Section \ref{hypoelliptic} to prove Theorem \ref{18032020T1}.

\begin{lem}\label{24042020L1} For all unbounded operators $A_1,A_2$ on $L^2(\mathbb R^n)$, respectively formally selfadjoint and formally skew-selfadjoint, both stabilizing the Schwartz space $\mathscr S(\mathbb R^n)$, we have
$$\forall u\in\mathscr S(\mathbb R^n),\quad\Reelle\big\langle A_1u,A_2u\big\rangle_{L^2} = \frac12\big\langle[A_1,A_2]u,u\big\rangle_{L^2},$$
with $[A_1,A_2] = A_1A_2-A_2A_1$ the commutator of the operators $A_1$ and $A_2$, which is well-defined on $\mathscr S(\mathbb R^n)$.
\end{lem}

\begin{proof} The result stated in Lemma \ref{24042020L1} ensues from a straightforward computation, since we have that for all $u\in\mathscr S$,
\begin{align*}
	\big\langle[A_1,A_2]u,u\big\rangle_{L^2} & = \big\langle A_1A_2u,u\big\rangle_{L^2} - \big\langle A_2A_1u,u\big\rangle_{L^2}
	= \big\langle A_2u,A_1u\big\rangle_{L^2} + \big\langle A_1u,A_2u\big\rangle_{L^2} \\[5pt]
	& = \overline{\big\langle A_1u,A_2u\big\rangle}_{L^2} + \big\langle A_1u,A_2u\big\rangle_{L^2}
	= 2\Reelle\big\langle A_1u,A_2u\big\rangle_{L^2}.
\end{align*}
\end{proof}

\subsection{The regularity exponents.}\label{regexp} We consider $r\geq2$ a positive integer and $\lambda_0>0$, $0\le q_{r-2}\le q_{r-1}\le1$ and $s_{r-1}\geq s_{r-2}\geq0$ some non-negative real numbers. By setting $q_j = s_j = 0$ when $j\le r-3$, we define recursively the positive real numbers $\lambda_1,\ldots,\lambda_r>0$ by
\begin{equation}\label{08042020E3}
	\forall j\in\{0,\ldots,r-1\},\quad\lambda_{j+1}\bigg(\frac{1-q_j}{\lambda_j}+1\bigg) = 1+s_j.
\end{equation}
The aim of this subjection is to state some properties of the family $(\lambda_0,\ldots,\lambda_r)$. More precisely, we aim at giving a more explicit formulation of the different assumptions of Theorem \ref{18032020T1}. We begin by noticing from an easy induction that
$$\lambda_j = \frac{\lambda_0}{1+j\lambda_0},\quad 0\le j\le r-2,\quad \lambda_{r-1} = \frac{\lambda_0(1+s_{r-2})}{1-q_{r-2}+\lambda_0+\lambda_0(1-q_{r-2})(r-2)},$$
and
$$\lambda_r = \frac{\lambda_0(1+s_{r-2})(1+s_{r-1})}{(1-q_{r-2})(1-q_{r-1}) + \lambda_0(1+s_{r-2}) + \lambda_0(1-q_{r-1}) + \lambda_0(1-q_{r-2})(1-q_{r-1})(r-2)}.$$
The proofs of the following lemmas are straightforward calculus that are omitted here.

\begin{lem}\label{10042019L1} The following equivalences
$$\lambda_{r-2}<\lambda_{r-1} \Leftrightarrow (s_{r-2} + q_{r-2})(1+\lambda_0(r-2))<\lambda_0,$$
and
\begin{align*}
	\lambda_{r-1}<\lambda_r & \Leftrightarrow (s_{r-1}+q_{r-1})(1-q_{r-2})(1+\lambda_0(r-2))<\lambda_0(1+s_{r-2}-s_{r-1}-q_{r-1}) \\[5pt]
	& \Leftrightarrow\lambda_r>q_{r-1}+s_{r-1}\geq q_{r-2}+s_{r-2},
\end{align*}
hold.
\end{lem}

\begin{lem} The following estimate holds
$$\frac{\lambda_{r-1}}{\lambda_{r-2}}+\frac{\lambda_{r-1}}{\lambda_r}\le 2,$$
if and only if
\begin{multline*}
	((1+s_{r-2})(1+s_{r-1})+(1-q_{r-2})(1-q_{r-1})-2(1+s_{r-1})(1-q_{r-2}))(1+\lambda_0(r-2)) \\[5pt]
	\le\lambda_0(2s_{r-1}-s_{r-2}+q_{r-1}).
\end{multline*}
\end{lem}

\subsection{Technical results} In this subsection, we give the proofs of some technical results used in Section \ref{constructionKalman} to prove Proposition \ref{05122018P1}. Let $B$ and $Q$ be real $n\times n$ matrices, with $Q$ symmetric positive semidefinite. We consider $r\geq1$ a positive integer and $\lambda_0>0$, $0\le q_{r-2}\le q_{r-1}\le1$ and $s_{r-1}\geq s_{r-2}\geq0$ some non-negative real numbers. By setting $q_j = s_j = 0$ when $j\le r-3$, we define recursively the positive real numbers $\lambda_1,\ldots,\lambda_r>0$ by \eqref{08042020E3}. Let us assume that the family $(\lambda_0,\ldots,\lambda_r)$ is decreasing and satisfies \eqref{26062020E2}.

\begin{lem}\label{29032019L1} Let $\chi\in C^{\infty}(\mathbb R)$ be a smooth function such that $\chi'$ is compactly supported with $\max\Supp\chi'>0$. There exists a positive constant $c>0$ such that for all $\xi\in\mathbb R^n$,
$$\bigg\vert B^T\xi\cdot\nabla_{\xi}\bigg(\chi\bigg(\frac{\vert Q(B^T)^{r-1}\xi\vert^2}{\langle\xi\rangle^{2\lambda_r/\lambda_{r-1}}}\bigg)\bigg)\bigg\vert
\le c\langle\xi\rangle^{1-\lambda_r/\lambda_{r-1}}\bigg\vert\chi'\bigg(\frac{\vert Q(B^T)^{r-1}\xi\vert^2}{\langle\xi\rangle^{2\lambda_r/\lambda_{r-1}}}\bigg)\bigg\vert.$$
\end{lem}

\begin{proof} It follows from a direct computation that for all $\xi\in\mathbb R^n$,
\begin{multline}\label{23072019E3}
	B^T\xi\cdot\nabla_{\xi}\bigg(\chi\bigg(\frac{\vert Q(B^T)^{r-1}\xi\vert^2}{\langle\xi\rangle^{2\lambda_r/\lambda_{r-1}}}\bigg)\bigg)
	= \frac{B^T\xi\cdot\nabla_{\xi}\vert Q(B^T)^{r-1}\xi\vert^2}{\langle\xi\rangle^{2\lambda_r/\lambda_{r-1}}}\chi'\bigg(\frac{\vert Q(B^T)^{r-1}\xi\vert^2}{\langle\xi\rangle^{2\lambda_r/\lambda_{r-1}}}\bigg) \\
	+ \vert Q(B^T)^{r-1}\xi\vert^2\big(B^T\xi\cdot\nabla_{\xi}\langle\xi\rangle^{-2\lambda_r/\lambda_{r-1}}\big)\chi'\bigg(\frac{\vert Q(B^T)^{r-1}\xi\vert^2}{\langle\xi\rangle^{2\lambda_r/\lambda_{r-1}}}\bigg).
\end{multline}
We need to compute and estimate the two terms in the above equality. On the one hand, we get that for all $\xi\in\mathbb R^n$,
\begin{align*}
	\frac{B^T\xi\cdot\nabla_{\xi}\vert Q(B^T)^{r-1}\xi\vert^2}{\langle\xi\rangle^{2\lambda_r/\lambda_{r-1}}}\chi'\bigg(\frac{\vert Q(B^T)^{r-1}\xi\vert^2}{\langle\xi\rangle^{2\lambda_r/\lambda_{r-1}}}\bigg)
	& = \frac{B^T\xi\cdot2B^{r-1}Q^2(B^T)^{r-1}\xi}{\langle\xi\rangle^{2\lambda_r/\lambda_{r-1}}}\chi'\bigg(\frac{\vert Q(B^T)^{r-1}\xi\vert^2}{\langle\xi\rangle^{2\lambda_r/\lambda_{r-1}}}\bigg) \\[5pt]
	& = \frac{2Q(B^T)^r\xi\cdot Q(B^T)^{r-1}\xi}{\langle\xi\rangle^{2\lambda_r/\lambda_{r-1}}}\chi'\bigg(\frac{\vert Q(B^T)^{r-1}\xi\vert^2}{\langle\xi\rangle^{2\lambda_r/\lambda_{r-1}}}\bigg).
\end{align*}
Cauchy Schwarz' inequality and the fact that the function $\chi$ is compactly supported with $\max\Supp\chi'>0$ therefore imply that for all $\xi\in\mathbb R^n$,
\begin{align}\label{23072019E1}
	&\ \bigg\vert\frac{B^T\xi\cdot\nabla_{\xi}\vert Q(B^T)^{r-1}\xi\vert^2}{\langle\xi\rangle^{2\lambda_r/\lambda_{r-1}}}\chi'\bigg(\frac{\vert Q(B^T)^{r-1}\xi\vert^2}{\langle\xi\rangle^{2\lambda_r/\lambda_{r-1}}}\bigg)\bigg\vert \\[5pt]
	\le &\ \frac{2\vert Q(B^T)^r\xi\vert\vert Q(B^T)^{r-1}\xi\vert}{\langle\xi\rangle^{2\lambda_r/\lambda_{r-1}}}\bigg\vert\chi'\bigg(\frac{\vert Q(B^T)^{r-1}\xi\vert^2}{\langle\xi\rangle^{2\lambda_r/\lambda_{r-1}}}\bigg)\bigg\vert, \nonumber \\[5pt]
	\lesssim &\ \frac{\langle\xi\rangle^{1+\lambda_r/\lambda_{r-1}}}{\langle\xi\rangle^{2\lambda_r/\lambda_{r-1}}}\bigg\vert\chi'\bigg(\frac{\vert Q(B^T)^{r-1}\xi\vert^2}{\langle\xi\rangle^{2\lambda_r/\lambda_{r-1}}}\bigg)\bigg\vert
	= \langle\xi\rangle^{1-\lambda_r/\lambda_{r-1}}\bigg\vert\chi'\bigg(\frac{\vert Q(B^T)^{r-1}\xi\vert^2}{\langle\xi\rangle^{2\lambda_r/\lambda_{r-1}}}\bigg)\bigg\vert. \nonumber
\end{align}
On the other hand, we compute the derivative of the following Japanese bracket for all $\xi\in\mathbb R^n$,
$$B^T\xi\cdot\nabla_{\xi}\langle\xi\rangle^{-2\lambda_r/\lambda_{r-1}} = \frac{-2\lambda_r}{\lambda_{r-1}}\langle\xi\rangle^{-2\lambda_r/\lambda_{r-1}-2}
B^T\xi\cdot\xi.$$
Exploiting the same ingredients as the ones used to obtain \eqref{23072019E1}, we get that for all $\xi\in\mathbb R^n$,
\begin{align}\label{23072019E2}
	&\ \bigg\vert\vert Q(B^T)^{r-1}\xi\vert^2\big(B^T\xi\cdot\nabla_{\xi}\langle\xi\rangle^{-2\lambda_r/\lambda_{r-1}}\big)\chi'\bigg(\frac{\vert Q(B^T)^{r-1}\xi\vert^2}{\langle\xi\rangle^{2\lambda_r/\lambda_{r-1}}}\bigg)\bigg\vert \\[5pt]
	= &\ \bigg\vert\frac{-2\lambda_r}{\lambda_{r-1}}\frac{\vert Q(B^T)^{r-1}\xi\vert^2B^T\xi\cdot\xi}{\langle\xi\rangle^{2\lambda_r/\lambda_{r-1}+2}}\chi'\bigg(\frac{\vert Q(B^T)^{r-1}\xi\vert^2}{\langle\xi\rangle^{2\lambda_r/\lambda_{r-1}}}\bigg)\bigg\vert \nonumber \\[5pt]
	\lesssim &\ \frac{\langle\xi\rangle^{2\lambda_r/\lambda_{r-1}}\langle\xi\rangle^2}{\langle\xi\rangle^{2\lambda_r/\lambda_{r-1}+2}}\bigg\vert\chi'\bigg(\frac{\vert Q(B^T)^{r-1}\xi\vert^2}{\langle\xi\rangle^{2\lambda_r/\lambda_{r-1}}}\bigg)\bigg\vert = \bigg\vert\chi'\bigg(\frac{\vert Q(B^T)^{r-1}\xi\vert^2}{\langle\xi\rangle^{2\lambda_r/\lambda_{r-1}}}\bigg)\bigg\vert. \nonumber
\end{align}
Gathering \eqref{23072019E3} and the estimates \eqref{23072019E1} and \eqref{23072019E2}, we deduce that for all $\xi\in\mathbb R^n$,
\begin{align*}
	\bigg\vert B^T\xi\cdot\nabla_{\xi}\bigg(\chi\bigg(\frac{\vert Q(B^T)^{r-1}\xi\vert^2}{\langle\xi\rangle^{2\lambda_r/\lambda_{r-1}}}\bigg)\bigg)\bigg\vert
	& \lesssim\big(1+\langle\xi\rangle^{1-\lambda_r/\lambda_{r-1}}\big)\bigg\vert\chi'\bigg(\frac{\vert Q(B^T)^{r-1}\xi\vert^2}{\langle\xi\rangle^{2\lambda_r/\lambda_{r-1}}}\bigg)\bigg\vert \\[5pt]
	& \lesssim\langle\xi\rangle^{1-\lambda_r/\lambda_{r-1}}\bigg\vert\chi'\bigg(\frac{\vert Q(B^T)^{r-1}\xi\vert^2}{\langle\xi\rangle^{2\lambda_r/\lambda_{r-1}}}\bigg)\bigg\vert,
\end{align*}
since $\lambda_{r-1}>\lambda_r$ by assumption. This ends the proof of Lemma \ref{29032019L1}.
\end{proof}

\begin{lem}\label{25072019L1} Let us assume that $r\geq2$. We consider the functions $W_k$, $\Psi_k$ and $\mathcal W_k$ defined in \eqref{25032019E6}, \eqref{25032019E9}, \eqref{27032019E6} and \eqref{24072019E18}. Then, for all $1\le k\le r-1$, the two following estimates hold for all $\xi\in\mathbb R^n$,
$$\big\vert\mathcal W_k(\xi)B^T\xi\cdot\nabla_{\xi}\Psi_{k+1}(\xi)\big\vert\lesssim_{\ \Gamma_k,\Gamma_{k+1}}\big\vert Q(B^T)^{r-k-1}\xi\big\vert^{\lambda_{r-k-1}}\mathcal W_{k+1}(\xi),$$
and
$$\big\vert\mathcal W_k(\xi)B^T\xi\cdot\nabla_{\xi}W_{k+1}(\xi)\big\vert\lesssim_{\ \Gamma_k,\Gamma_{k+1}}\big\vert Q(B^T)^{r-k-1}\xi\big\vert^{\lambda_{r-k-1}}\mathcal W_k(\xi)\Psi_{k+1}(\xi),$$
with a dependence of the constants only in $\Gamma_2$ in the case where $k=1$ by convention.
\end{lem}

\begin{proof} Let $1\le k\le r-1$ fixed. We first focus on proving that for all $\xi\in\mathbb R^n$,
\begin{equation}\label{25072019E12}
	\big\vert\mathcal W_k(\xi)B^T\xi\cdot\nabla_{\xi}\Psi_{k+1}(\xi)\big\vert\lesssim_{\ \Gamma_k,\Gamma_{k+1}}\big\vert Q(B^T)^{r-k-1}\xi\big\vert^{\lambda_{r-k-1}}\mathcal W_{k+1}(\xi).
\end{equation}
Let us recall from \eqref{27032019E6} that the symbol $\Psi_{k+1}(\xi)$ is defined on $\Supp \mathcal W_k$ by
$$\Psi_{k+1}(\xi) = \psi\bigg(\frac{\Gamma^2_{k+1}\vert Q(B^T)^{r-k-1}\xi\vert^2}{\vert Q(B^T)^{r-k}\xi\vert^{2\lambda_{r-k}/\lambda_{r-k-1}}}\bigg).$$ 
A direct calculus shows that for all $\xi\in\Supp(\mathcal W_k\Psi_{k+1})$,
$$B^T\xi\cdot\nabla_{\xi}\vert Q(B^T)^{r-k-1}\xi\vert^2 = 2Q(B^T)^{r-k}\cdot Q(B^T)^{r-k-1}\xi,$$
and
$$B^T\xi\cdot\nabla_{\xi}\vert Q(B^T)^{r-k}\xi\vert^{-2\lambda_{r-k}/\lambda_{r-k-1}} = \frac{-2\lambda_{r-k}}{\lambda_{r-k-1}}\frac{Q(B^T)^{r-k+1}\xi\cdot Q(B^T)^{r-k}\xi}{\vert Q(B^T)^{r-k}\xi\vert^{2\lambda_{r-k}/\lambda_{r-k-1}+2}}.$$
We therefore deduce that for all $\xi\in\mathbb R^n$,
$$\mathcal W_k(\xi)B^T\xi\cdot\nabla_{\xi}\Psi_{k+1}(\xi) = A_{1,k}(\xi) + A_{2,k}(\xi),$$ 
where the two terms $A_{1,k}(\xi)$ and $A_{2,k}(\xi)$ are respectively defined by
$$A_{1,k}(\xi) = \frac{\Gamma^2_{k+1}Q(B^T)^{r-k}\xi\cdot Q(B^T)^{r-k-1}\xi}{\vert Q(B^T)^{r-k}\xi\vert^{2\lambda_{r-k}/\lambda_{r-k-1}}}\ \psi'\bigg(\frac{\Gamma^2_{k+1}\vert Q(B^T)^{r-k-1}\xi\vert^2}{\vert Q(B^T)^{r-k}\xi\vert^{2\lambda_{r-k}/\lambda_{r-k-1}}}\bigg)\mathcal W_k(\xi),$$
and 
\begin{multline*}
	A_{2,k}(\xi) = \frac{-2\lambda_{r-k}}{\lambda_{r-k-1}}\frac{\Gamma^2_{k+1}\vert Q(B^T)^{r-k-1}\xi\vert^2Q(B^T)^{r-k+1}\xi\cdot Q(B^T)^{r-k}\xi}{\vert Q(B^T)^{r-k}\xi\vert^{2\lambda_{r-k}/\lambda_{r-k-1}+2}} \\[5pt]
	\times\psi'\bigg(\frac{\Gamma^2_{k+1}\vert Q(B^T)^{r-k-1}\xi\vert^2}{\vert Q(B^T)^{r-k}\xi\vert^{2\lambda_{r-k}/\lambda_{r-k-1}}}\bigg)\mathcal W_k(\xi).
\end{multline*}
We now aim at controlling $A_{1,k}(\xi)$ and $A_{2,k}(\xi)$. Before that, notice from the definitions \eqref{23072019E11} and \eqref{24072019E4} of the functions $w$ and $\psi$ that the inequality $\vert\psi'\vert\lesssim w$ holds, which implies that for all $\xi\in\mathbb R^n$,
\begin{multline*}
	\bigg\vert\psi'\bigg(\frac{\Gamma^2_{k+1}\vert Q(B^T)^{r-k-1}\xi\vert^2}{\vert Q(B^T)^{r-k}\xi\vert^{2\lambda_{r-k}/\lambda_{r-k-1}}}\bigg)\bigg\vert\mathcal W_k(\xi) \\[5pt]
	\lesssim w\bigg(\frac{\Gamma^2_{k+1}\vert Q(B^T)^{r-k-1}\xi\vert^2}{\vert Q(B^T)^{r-k}\xi\vert^{2\lambda_{r-k}/\lambda_{r-k-1}}}\bigg)\mathcal W_k(\xi)
	=W_{k+1}(\xi) \mathcal W_k(\xi) = \mathcal W_{k+1}(\xi).
\end{multline*}
In order to estimate $A_{1,k}(\xi)$, we use the definition \eqref{24072019E4} of the function $\psi$, the above estimate and Cauchy-Schwarz' inequality to obtain that for all $\xi\in\mathbb R^n$,
$$\big\vert A_{1,k}(\xi)\big\vert\lesssim_{\ \Gamma_{k+1}}\frac{\vert Q(B^T)^{r-k}\xi\vert^{1+\lambda_{r-k}/\lambda_{r-k-1}}}{\vert Q(B^T)^{r-k}\xi\vert^{2\lambda_{r-k}/\lambda_{r-k-1}}}\ \mathcal W_{k+1}(\xi).$$
Moreover, we deduce from \eqref{08042020E3} that 
$$1-\frac{\lambda_{r-k}}{\lambda_{r-k-1}} = \lambda_{r-k}-s_{r-k-1}-\frac{q_{r-k-1}\lambda_{r-k}}{\lambda_{r-k-1}}\le\lambda_{r-k},$$
As a consequence of this estimate and \eqref{08042020E4}, we get that for all $\xi\in\mathbb R^n$,
$$\big\vert A_{1,k}(\xi)\big\vert\lesssim_{\ \Gamma_{k+1}}\vert Q(B^T)^{r-k}\xi\vert^{\lambda_{r-k}}\ \mathcal W_{k+1}(\xi)\lesssim_{\ \Gamma_{k+1}}\vert Q(B^T)^{r-k-1}\xi\vert^{\lambda_{r-k-1}}\ \mathcal W_{k+1}(\xi).$$
The same ingredients, combined with the use of the estimate
$$\frac{\lambda_{r-k}}{\lambda_{r-k+1}}-1\le\lambda_{r-k},$$
which is a consequence of the assumption \eqref{26062020E2},
$$\frac{\lambda_{r-k}}{\lambda_{r-k+1}} + \frac{\lambda_{r-k}}{\lambda_{r-k-1}} = \frac{\lambda_{r-k}}{\lambda_{r-k+1}} + 1 - \lambda_{r-k} + s_{r-k-1} + \frac{q_{r-k-1}\lambda_{r-k}}{\lambda_{r-k-1}}\le2,$$
allow similarly to control $A_{2,k}(\xi)$ for all $\xi\in\mathbb R^n$ in the following way
\begin{align*}
	\big\vert A_{2,k}(\xi)\big\vert & \lesssim_{\ \Gamma_k,\Gamma_{k+1}}\frac{\vert Q(B^T)^{r-k}\xi\vert^{\lambda_{r-k}/\lambda_{r-k+1}}}{\vert Q(B^T)^{r-k}\xi\vert}\ \mathcal W_{k+1}(\xi) \\[5pt]
	& \lesssim_{\ \Gamma_k,\Gamma_{k+1}}\vert Q(B^T)^{r-k-1}\xi\vert^{\lambda_{r-k-1}}\ \mathcal W_{k+1}(\xi).
\end{align*}
The two above inequalities imply that the estimate \eqref{25072019E12} actually holds. Finally, the estimate 
$$\big\vert\mathcal W_k(\xi)B^T\xi\cdot\nabla_{\xi}W_{k+1}(\xi)\big\vert\lesssim_{\ \Gamma_k,\Gamma_{k+1}}\big\vert Q(B^T)^{r-k-1}\xi\big\vert^{\lambda_{r-k-1}}\mathcal W_k(\xi)\Psi_{k+1}(\xi),$$
valid for all $\xi\in\mathbb R^n$, can be obtained in the very same way, since the function $\psi'$ and $w'$ have the same supports and by using that the inequality $\vert w'\vert\lesssim\psi$ holds. This ends the proof of Lemma \ref{25072019L1}.
\end{proof}

\end{document}